\documentclass[12pt,oneside,a4paper,reqno]{amsart}
\pdfoutput=1

\usepackage[utf8]{inputenc}
\usepackage[T1]{fontenc}
\usepackage{lmodern}
\usepackage{microtype}

\usepackage[hmargin=2.5cm,vmargin=2.5cm]{geometry}
\usepackage[foot]{amsaddr}
\usepackage{mathtools}
\usepackage{amssymb}
\usepackage{aligned-overset}
\usepackage{bbm}

\usepackage[sortcites,giveninits=true,eprint=false,url=false,isbn=false,maxnames=99]{biblatex}
\usepackage{amsthm}
\usepackage{url}
\usepackage{hyperref}
\hypersetup{breaklinks=true,colorlinks,allcolors=blue,bookmarksopen,
            pdftitle={Noncommutative orbital stability of stochastic patterns in Banach spaces},
            pdfauthor={Joris van Winden}}
\usepackage[capitalise]{cleveref} 
\numberwithin{equation}{section}
\crefname{equation}{}{}
\crefname{assumption}{Assumption}{Assumptions}

\newtheorem{theorem}{Theorem}[section]
\newtheorem{lemma}[theorem]{Lemma}
\newtheorem{assumption}{Assumption}
\newtheorem{proposition}[theorem]{Proposition}
\theoremstyle{definition}
\newtheorem{definition}[theorem]{Definition}
\theoremstyle{remark}
\newtheorem{remark}[theorem]{Remark}

\newcommand{\cC}{{\mathcal C}}
\newcommand{\cD}{{\mathcal D}}
\newcommand{\cF}{{\mathcal F}}
\newcommand{\cG}{{\mathcal G}}
\newcommand{\cH}{{\mathcal H}}
\newcommand{\cL}{{\mathcal L}}
\newcommand{\cO}{{\mathcal O}}
\newcommand{\cP}{{\mathcal P}}
\newcommand{\cT}{{\mathcal T}}
\newcommand{\cV}{{\mathcal V}}
\newcommand{\cW}{{\mathcal W}}
\newcommand{\cX}{{\mathcal X}}
\newcommand{\cY}{{\mathcal Y}}

\newcommand{\bbC}{{\mathbb C}}
\newcommand{\bbN}{{\mathbb N}}
\newcommand{\bbP}{{\mathbb P}}
\newcommand{\bbR}{{\mathbb R}}

\newcommand{\rd}{{\, \rm{d}}}
\newcommand{\eps}{\varepsilon}
\newcommand{\ind}{\mathbbm{1}}

\DeclareMathOperator{\ad}{ad}
\DeclareMathOperator{\Real}{Re}

\DeclarePairedDelimiter{\abs}{\lvert}{\rvert}
\DeclarePairedDelimiter{\norm}{\lVert}{\rVert}
\DeclarePairedDelimiter{\cur}{\{}{\}}
\DeclarePairedDelimiter{\bra}{(}{)}
\DeclarePairedDelimiter{\sqr}{[}{]}

\newcommand{\PP}[2][]{\mathbb{P}\, \sqr[#1]{#2}}
\newcommand{\EE}[2][\big]{\mathbb{E} \sqr[#1]{#2}}

\newcommand{\grp}{{\mathrm{G}}}
\newcommand{\alg}{{\mathfrak{g}}}
\newcommand{\eye}{{\mathrm I}}

\title{Noncommutative orbital stability of stochastic patterns in Banach spaces}
\author{Joris van Winden}
\date{June 24, 2024}
\address{Delft Institute of Applied Mathematics, Faculty of Electrical Engineering, Mathematics and Computer Science, Delft University of Technology, Mekelweg 4, 2628 CD Delft, Netherlands}
\email{J.vanWinden@tudelft.nl}
\thanks{The author thanks Manuel Gnann for careful reading of the manuscript. Part of this work is based on the MSc thesis of the author, prepared under supervision of Manuel Gnann at Delft University of Technology. The author is supported by a DIAM fast-track scholarship}

\begin{document}

\begin{abstract}
We consider stochastic perturbations of PDEs which have special pattern solutions, such as (nonlinear) travelling waves, solitons, and spiral waves.
We show orbital stability of these patterns on a timescale which is exponential in the inverse square of the noise amplitude.
We systematically treat equations with noncommutative symmetry groups, and show how the noncommutativity affects the motion of the pattern.
This is done by introducing a new method to track the (generalized) phase of the pattern.
Furthermore, we demonstrate how orbital stability arises from a mismatch of symmetry between the pattern and the equation.
Our phase tracking method does not rely on a Hilbert space structure.
This allows us to show stability in general Banach spaces, and to treat noise with lower regularity than before.
\end{abstract}

\maketitle

\section{Introduction}
Recently, a significant interest has developed in studying the stability of patterns in stochastic partial differential equations 
(SPDEs) \cite{inglis_general_2016,maclaurin_phase_2023,kuehn_stochastic_2022,liu_stability_2023,hamster_stability_2019,hamster_stability_2020,hamster_stability_2020a,hamster_travelling_2020,lang_multiscale_2016,gnann_solitary_2024,kruger_front_2014,kruger_multiscaleanalysis_2017,adams_isochronal_2022,cartwright_collective_2019,cartwright_collective_2021,vandenbosch_multidimensional_2024}.
Commonly studied patterns include travelling waves, travelling pulses, spiral waves, solitary waves, and solitons.
A typical feature is that these patterns exhibit \emph{orbital stability}, meaning that the solution to the SPDE remains close to a suitably shifted version of the pattern.
The exact nature of this shift, which we will henceforth refer to as the \emph{phase} or \emph{phase shift}, depends on the geometry of the equation.
To show orbital stability, it is often necessary to have a method to continuously track the phase of the pattern, for which various established methods are available \cite{kruger_front_2014,hamster_stability_2019,inglis_general_2016,cartwright_collective_2019}.
However, these methods typically rely on orthogonality conditions, which are only available when working in a Hilbert space.
In this paper, we make the following contribution to this field of research:
\begin{itemize}
    \item We explain how orbital stability arises through symmetry, and show what kind of phase shift is expected for a given pattern.
    \item We introduce a new method of tracking the phase, which does not rely on Hilbert space geometry.
    \item We give explicit expressions to compute the phase, which are valid for patterns with noncommutative symmetry groups.
\end{itemize}
The main results (\cref{thm:stochstabshort,thm:stochstablong}) show orbital stability of stochastically forced patterns in Banach spaces, on a timescale which is exponential in the inverse square of the noise amplitude.
We also directly relate the orbital stability to the symmetry group of the equation.
The main novelties of this work are that we treat a general noncommutative setting, and show stability without assuming an underlying Hilbert space structure.
The advantage of working in a Banach space setting is that we can allow for rougher noise, as is demonstrated \cref{subsec:fhn}.

\subsection{Orbital stability and symmetry}
The prototypical example of an orbitally stable pattern is that of a travelling (nonlinear) wave or pulse.
The literature on these waves is vast, and it is not feasible to give a comprehensive overview here.
A (nonexhaustive) list of settings in which these waves have been studied consists of hydrodynamics \cite{korteweg_change_1895}, neural field equations \cite{conley_application_1984}, fiber optics equations \cite{marcq_exact_1994,maimistov_solitons_2010}, and predator-prey models \cite{gardner_existence_1983}.
For more comprehensive treatments of this topic, we refer the reader to \cite{sandstede_stability_2002,kapitula_spectral_2013,volpert_traveling_1994a,kuehn_travelling_2020}.

The mathematical treatment of these waves can be subdivided into three somewhat separate aspects: existence, linear stability, and nonlinear stabiliy.
A prime example of each of these aspects is found in the seminal works of Evans on nonlinear waves and pulses in neural field equations \cite{evans_nerve_1972,evans_nerve_1972a,evans_nerve_1972b,evans_nerve_1975}.
The treatment of stochastically perturbed travelling waves is much more recent (see e.g.\ \cite{kruger_front_2014,hamster_stability_2019,inglis_general_2016}).
However, the earliest work treating orbital stability in an SPDE that we are aware of is \cite{katzenberger_solutions_1991}.
As our primary goal is to treat nonlinear stochastic stability, we take existence and linear stability of a pattern for granted (see \cref{ass:selfsimilarsolution} in \cref{sec:linsym}).

In the case of a travelling wave, it may seem obvious that a translational correction is the `right' way to shift the pattern.
However, in higher dimensions, the situation is not as clear.
A primary motivating example in this situation is a two-dimensional spiral wave, which requires a phase shift consisting of translations and rotations \cite{beyn_nonlinear_2008,kuehn_stochastic_2022}.
At first glance, one may wonder why a rotational correction does not suffice to show stability.
In fact, even in the case of a travelling wave, it is not immediately obvious why a phase correction is necessary in the first place.
In \cref{sec:linsym}, we answer these questions by stating the following principle: 

\begin{quote}
\emph{Orbital stability arises from continuous symmetries of the equation which are not shared by the pattern.}
\end{quote}

This immediately clarifies the origin and nature of the phase correction for the spiral wave and the travelling waves.
Furthermore, it provides a guide to determine how a pattern with a more complicated symmetry group is expected to move.

Systematic treatments of patterns with bigger symmetry groups have been given before (see \cite[Chapter 4.2]{kapitula_spectral_2013}, \cite{maclaurin_phase_2023}).
However, both works contain the explicit or implicit assumption that the symmetry group is commutative, which is a significant limitation.
Note that the recent work \cite{adams_isochronal_2022}, which does not seem to include commutativity assumptions, defers the stability proof to \cite{maclaurin_phase_2023}, in which commutativity is implicitly assumed.
The commutativity already poses a problem when treating the two-dimensional spiral wave, as the relevant symmetry group of $\bbR^2$ (consisting of rotations and translations) is noncommutative.
It can be seen in \cite{beyn_nonlinear_2008,kuehn_stochastic_2022} that this noncommutativity plays a significant role in the analysis.
In the noncommutative setting, there is the work of \citeauthor{beyn_freezing_2004} \cite{beyn_freezing_2004} dealing with PDEs with continuous symmetries.
Although the setting is similar to ours, \cite{beyn_freezing_2004} does not address the matter of orbital stability of patterns.
Moreover, there are serious analytical challenges when one tries to adapt the `freezing method' to a stochastic setting, as can be seen in e.g.\ \cite{vandenbosch_multidimensional_2024}.

Using the basic theory of Lie groups and Lie algebras, we give a systematic treatment of (stochastic) nonlinear stability which is valid in the noncommutative case.
We show that the linearized dynamics around the pattern can be explicitly described in terms of the Lie algebra corresponding to the symmetry group of the equation (see \cref{subsec:dynamicscomoving}).
Moreover, from our method of phase tracking, it can directly be seen how noncommutativity affects the motion of the pattern (see \cref{subsec:orbstab}).

\subsection{Phase tracking}
\label{subsec:introphasetrack}
When showing orbital stability of stochastically perturbed patterns, a crucial aspect is the issue of how to track the phase.
In the recent literature, several different ways of accomplishing this have been formulated.
We identify the following methods of phase tracking:
\begin{itemize}
    \item The variational phase \cite{inglis_general_2016,maclaurin_phase_2023,kuehn_stochastic_2022,liu_stability_2023},
    \item The stochastic freezing phase \cite{hamster_stability_2019,hamster_stability_2020,hamster_stability_2020a,hamster_travelling_2020},
    \item The phase-lag method \cite{lang_multiscale_2016,gnann_solitary_2024,kruger_front_2014,kruger_multiscaleanalysis_2017},
    \item The isochronal phase \cite{adams_isochronal_2022,adams_asymptotic_2023}.
\end{itemize}
Note that any two `valid' notions of phase must be closely related, as the pattern can only be in one location at a time.

Despite the conceptual variety of these phase tracking methods, the cited works all rely in some way or another on a Hilbert space structure, which significantly limits the applicability and poses restrictions on the noise.
The variational phase and stochastic freezing phase are both defined in terms of orthogonality conditions, so they are not well-defined outside of a Hilbert space.
The phase-lag method and isochronal phase seem more suitable, with definitions which (partially) generalize to Banach spaces.
However, the associated stability proofs still rely in a nontrivial way on the presence of an inner product.

In \cref{sec:nonlinstab}, we introduce a new method of phase tracking, which we call the \emph{predicted phase} \cref{eq:predictedphase}.
It is defined using a decomposition of the initial profile, and can be viewed as a first-order approximation to the isochronal phase.
In the case of a travelling pulse solution of the form $u(t,x) = u^*(x-t)$, where $u^*$ is the pulse profile, it can be described as follows.
Consider an initial condition of the form $u(0,\cdot) = u^* + v_0$, where $v_0$ is a perturbation of the profile which is $\cO(\eps)$.
If the pulse is linearly stable (see \cref{ass:decomposition}), we can find a decomposition 
\begin{equation}
    \label{eq:introv0decomp}
    v_0 = a \partial_x u^* + w_0,
\end{equation}
where $a,w_0$ are both $\cO(\eps)$, and $w_0$ decays exponentially under the linear dynamics.
For more details on how to compute this decomposition, we refer ahead to \cref{subsec:stabilitycomoving}.
By \cref{eq:introv0decomp} and a Taylor expansion, we have
\begin{align*}
    u(0,\cdot) &= u^* + a \partial_x u^* + w_0 \\
               &= u^*(\cdot + a) + w_0 + \cO(\eps^2).
\end{align*}
Using the exponential decay of $w_0$ under the linear dynamics, and treating the nonlinear and stochastic terms perturbatively, we obtain for $t\geq0$ the expansion
\begin{equation*}
    u(t,\cdot) = u^*(\cdot + a) + \cO(\eps e^{-t}) + \cO(\eps^2).
\end{equation*}
For a fixed, large enough $T$, and sufficiently small $\eps$, we thus find that the difference $u(T,\cdot) - u^*(\cdot - a)$ is smaller (by a constant factor) than the initial difference $u(0,\cdot) - u^*$, when measured in suitable norms.
We then repeat this procedure on the time intervals $[T,2T]$, $[2T,3T]$ and so on to obtain stability on long timescales.

Because of the stochastic forcing, there is on each time interval $[nT,(n+1)T]$ a nonzero probability $p$ that the solution strays too far from the stable manifold.
Using the subgaussian tail estimates on stochastic convolutions formulated in \cref{subsec:tail}, we will estimate $p \lesssim \exp(-c \eps^2 \sigma^{-2})$ for some $c > 0$, where $\sigma$ denotes the noise amplitude (see \cref{eq:stochstabshort}).
The probability $p_\eps(t)$ that the solution leaves an $\eps$-neighbourhood of the stable manifold before time $t$ can then be estimated as $p_\eps(t) \lesssim t\exp(-c\eps^2 \sigma^{-2})$.
Thus, for small noise amplitude ($\sigma \ll 1$), the pattern is stable for a long time with high probability.
We refer ahead to \cref{subsec:orbstab} for the definition of the phase for patterns with more complicated (noncommutative) symmetries.

Our method has two advantages compared to the previously mentioned ones.
The first is that the expression for the phase \cref{eq:predictedphase} is explicit and straightforward to compute (for two examples, see \cref{subsec:rotwave,rem:fhnphase}).
The phase is directly determined from the initial condition, and there is no need to couple an SDE to the SPDE to continuously track the phase.
This also bypasses analytical challenges which are present in previous works, and allows for a concise stability proof, as can be seen in \cref{sec:proof}.

Secondly, the predicted phase is defined without any reference to a Hilbert space structure.
In fact, all the linear stability assumptions in \cref{sec:linsym} are formulated in a general Banach space,
as are the (stochastic) nonlinear stability results in \cref{sec:nonlinstab}.
To our knowledge, this is the first stochastic stability result in this setting, and the flexiblity afforded by this approach allows us to treat noise of various levels of regularity in the example in \cref{subsec:fhn}.
The Banach space setting is important for future applications, since $L^p$-theory with $p \neq 2$ has been proven to be effective in showing well-posedness of SPDEs.
For parabolic equations, recent advances in this area have been made using maximal regularity techniques \cite{agresti_nonlinear_2022a,agresti_reactiondiffusion_2023,agresti_nonlinear_2022}.

Our main stability results (\cref{thm:stochstabshort,thm:stochstablong}) match the best current results obtained with other phase tracking methods \cite{hamster_stability_2020a,maclaurin_phase_2023,gnann_solitary_2024}, and are expected to be optimal when stated in this generality.
We pose our assumptions in a way which is mostly agnostic to the analytical properties of the PDE: we do not assume any smoothing or dispersion, and only require the nonlinearity to be locally Lipschitz. 
We allow for additive noise and multiplicative noise in either the It\^o or Stratonovich sense.
As a consequence, our results apply to parabolic equations, dispersive equations (see \cite{westdorp_soliton_2024,gnann_solitary_2024}), and PDE-ODE systems such as the FitzHugh--Nagumo equation (see \cite{eichinger_multiscale_2022}).
Being formulated without specific knowledge of the equation, our results are generally not optimal in terms of regularity when specifying to any of these settings.
For example, we expect that in the parabolic setting, even lower regularity of the noise can be achieved by making use of the smoothing properties of the equation (see \cref{rem:fhnparabolic}).

The local Lipschitz conditions formulated in \cref{ass:Fregular,ass:GHregular} present an obstable to showing stability with exceedingly rough noise (see \cref{subsec:fhn}), but we only formulate these assumptions to be able to deal with a very general class of equations.
We expect that in many concrete situations, one can adapt our method to deal with the perturbative terms in a way which is taylored to the specific equation.
For some examples of such treatments, we refer to \cite{gnann_solitary_2024,eichinger_multiscale_2022,westdorp_soliton_2024}.

\subsection{Outline}
The outline of this paper is as follows.
In \cref{sec:prelim}, notation and preliminaries regarding stochastic integration are stated.
\cref{sec:linsym} discusses linear stability, and motivates and introduces the main assumption for our stability result.
Afterwards, we explain in \cref{sec:nonlinstab} how orbital stability arises from a nontrivial center space (see \cref{def:centerspace}).
We also introduce the predicted phase \cref{eq:predictedphase}, and formulate the main stability results (\cref{thm:detstab,thm:stochstabshort,thm:stochstablong}) for deterministic and stochastic perturbations.
In \cref{sec:examples} we revisit two examples from the literature: a travelling pulse in a fully diffuse FitzHugh--Nagumo equation \cite{alexander_topological_1990,hamster_stability_2020}, and a two-dimensional spiral wave in a reaction-diffusion equation \cite{beyn_nonlinear_2008,kuehn_stochastic_2022}.
For the FitzHugh--Nagumo pulse, we extend the results of \cite{hamster_stability_2020} to a wide range of noises with low regularity.
In the example of the spiral wave, we compute the predicted phase explicitly, and demonstrate how the noncommutativity of the symmetry enters into the phase.
Finally, \cref{sec:proof} contains the proofs of the main stability results.

\section{Preliminaries}
\label{sec:prelim}
\subsection{Notation}
\label{subsec:notation}
We use the convention that $\bbN$ is the set of nonnegative integers (including zero).
Throughout the paper, we let $(\Omega,\cF,\bbP)$ be a complete probability space, equipped with a complete and right-continuous filtration $\cur{\cF_t}_{t \geq 0}$.
When we speak of adaptedness or progressive measurability, it will be with respect to this filtration unless the contrary is explicitly stated.
We write $\mathbb{E}$ for the expectation associated with $\mathbb{P}$.
We write $\norm{\cdot}_\cX$ for the norm of a general real Banach space $\cX$.
The space of bounded linear operators between two Banach spaces $\cX$ and $\cY$ is denoted $\cL(\cX;\cY)$, and in the case $\cX = \cY$ we write $\cL(\cX) \coloneq \cL(\cX;\cX)$.
For an unbounded operator $A$ between $\cX$ and $\cY$, we denote its domain by $\cD(A)$ and the spectrum by $\sigma(A)$.
We use the notation $\cX \hookrightarrow \cY$ to mean that $\cX$ embeds continuously into $\cY$.

When $M$ is a metric space, we write $C(M;\cX)$ (resp.\ $C_{\mathrm{ub}}(M;\cX)$) for the space of continuous (resp.\ bounded uniformly continuous) functions from $M$ to $\cX$.
For a measure space $(S,\cG,\mu)$ and $p \in [1,\infty]$ we write $L^p(S;\cX)$ for the Lebesgue--Bochner space of strongly measurable $\cX$-valued functions which are $p$-integrable (or essentially bounded if $p = \infty$).
From now on, we will simply write measurable instead of strongly measurable.
Note that if $\cX$ is separable, the two notions are equivalent.
We abbreviate $L^p_{\Omega}(\cX) \coloneq L^p(\Omega;\cX)$ and $L^p(0,t;\cX) \coloneq L^p([0,t];\cX)$, where the latter is equipped with the Lebesgue measure.

For $k,d,n \in \bbN$ (with $d,n$ nonzero) and $p \in [1,\infty]$, we denote by $W^{k,p}(\bbR^d;\bbR^n)$ the classical Sobolev space of measurable functions from $\bbR^d$ to $\bbR^n$ which have $k$ weak derivatives which are $p$-integrable (or essentially bounded in the case $p = \infty$).
When $p \in (1,\infty)$, $s \in [0,\infty)$, we write $H^{s,p}(\bbR^d;\bbR^n)$ for the Bessel space defined via the norm $\norm{(\eye - \Delta)^{\frac{s}{2}} f}_{L^p(\bbR^d;\bbR^n)}$, where $(\eye - \Delta)^{\frac{s}{2}}$ is defined via the Fourier symbol $\xi \mapsto (1 + \abs{\xi}^2)^{\frac{s}{2}}$.
In the case $p=2$, we will write $H^{s}(\bbR^d;\bbR^n)$ instead of $H^{s,2}(\bbR^d;\bbR^n)$.

When $\cH,\cH'$, are Hilbert spaces, we write $\cL_2(\cH;\cH')$ for the closed subspace of $\cL(\cH;\cH')$ consisting of Hilbert--Schmidt operators.
The space of $\gamma$-radonifying operators from $\cH$ to $\cX$, denoted $\gamma(\cH;\cX)$, is defined as the closure of the finite rank operators $T \in \cL(\cH;\cX)$ with respect to the norm
\begin{equation*}
    \norm{T}_{\gamma(\cH;\cX)} = \sup \bra[\Big]{\mathbb{\tilde{E}}\sqr[\Big]{\norm[\big]{\sum_{j=1}^{n} \gamma_j T h_j}^2_\cX}}^{\frac{1}{2}},
\end{equation*}
where $(\gamma_j)_{j \geq 1}$ is a sequence of independent standard Gaussian random variables on some probability space $(\widetilde{\Omega},\widetilde{\cF},\widetilde{\mathbb P})$, $\widetilde{\mathbb E}$ denotes the expecation with respect to $\widetilde{\bbP}$,
and the supremum is taken over all sets of orthonormal vectors in $\cH$.

Finally, we write $[\cdot,\cdot] \colon \alg \times \alg \to \alg$ for the Lie bracket of a Lie algebra $\alg$.

\subsection{Stochastic integration and tail estimates}
\label{subsec:tail}
In this section, we let $\cH$ be a separable Hilbert space, and let $W(t)$ be an $\cH$-cylindrical Wiener process.

In \cref{subsec:stochpert} we will make use of the theory of stochastic integration in $2$-smooth Banach spaces.
For an introduction on this topic and further references, we refer the reader to \cite{vanneerven_stochastic_2015a}.
The condition that $\cX$ is $2$-smooth is necessary to have a theory of stochastic integration which is satisfactory for our purposes.
This condition is satisfied by many commonly used spaces, including $L^p$, $W^{k,p}$, and $H^{s,p}$ for $p \in [2,\infty)$.
We note that $2$-smoothness is generally not preserved under isomorphisms, so one must take care to use the `right' norm for these spaces.
However, the celebrated work of Pisier \cite{pisier_martingales_1975} shows that any space which has martingale type $2$ admits an equivalent $2$-smooth norm.

To show stability on long timescales, it will be necessary to have a subgaussian tail estimate for stochastic convolutions.
\cref{prop:stochconvtail} suffices for this purpose.
\begin{lemma}
    \label{lem:tail}
    Let $K > 0$, and let $X$ be a nonnegative random variable which satisfies
    \begin{equation}
        \label{eq:Xmoment}
        \EE[]{X^p} \leq \sqrt{p}^{\,p} K^p, \quad p \in [2,\infty).
    \end{equation}
    Then $X$ satisfies the subgaussian tail estimate
    \begin{equation}
        \label{eq:Xtail}
        \PP{X > \lambda} \leq e \exp\bra[\big]{-(2e)^{-1} \lambda^2 K^{-2}}, \quad \lambda \geq 0.
    \end{equation}
\end{lemma}
\begin{proof}
    Let $\lambda \geq 0$, and set $q = \lambda^2 K^{-2} e^{-1}$.
    If $q \in [0,2]$, then $e \exp(-(2e)^{-1} \lambda^2 K^{-2}) = e \exp(-2^{-1}q) \geq 1$, so \cref{eq:Xtail} is trivial.
    If $q \in [2,\infty)$, then we use Markov's inequality and \cref{eq:Xmoment} to find
    \begin{equation*}
        \PP{X > \lambda} \leq \lambda^{-q}\EE[]{X^q} \leq \lambda^{-q}\sqrt{q}^{\,q} K^q = \exp(-(2e)^{-1} \lambda^2 K^{-2}),
    \end{equation*}
    so \cref{eq:Xtail} is satisfied.
\end{proof}

\begin{proposition}
\label{prop:stochconvtail}
Let $\cX$ be a $2$-smooth Banach space, and let $\cur{S(t,t')}_{0 \leq t \leq t'}$ be a strongly continuous evolution family on $\cX$.
There exists a constant $c > 0$ such that the estimate
\begin{equation}
    \label{eq:stochconvtail}
    \PP[\Big]{\sup_{t \in [0,T]} \norm[\big]{\int_0^t S(t,t')f(t') \rd W(t')}_\cX \geq \lambda} 
    \leq e \exp\bra[\Big]{\frac{-c\lambda^2}{T\,\norm{f}_{L^{\infty}_{\Omega}(L^{\infty}(0,T;\gamma(\cH;\cX)))}^2}}
\end{equation}
holds for all $T > 0$, $\lambda \geq 0$, and all progressively measurable $f \in L^{\infty}_{\Omega}(L^{\infty}(0,T;\gamma(\cH;\cX)))$.
\end{proposition}
\begin{proof}
    Fix $T > 0$ and $f \in L^{\infty}_{\Omega}(L^{\infty}(0,T;\gamma(\cH;\cX)))$.
    By \cite[Theorem 4.5]{vanneerven_maximal_2020}, there is a constant $C$ (depending only on $\cX,S$) such that the estimate
    \begin{equation*}
        \norm[\Big]{\sup_{t \in [0,T]} \norm[\big]{\int_0^t S(t,t')f(t') \rd W(t')}_\cX}_{L^p_{\Omega}} 
        \leq C\sqrt{p}T^{\frac{1}{2} - \frac{1}{p}}\norm{f}_{L^{p}_{\Omega}(L^p(0,T;\gamma(\cH;\cX)))}
    \end{equation*}
    holds for all $p \in [4,\infty)$.
    Using H\"older's inequality, we find
    \begin{equation*}
        \norm[\Big]{\sup_{t \in [0,T]} \norm[\big]{\int_0^t S(t,t')f(t') \rd W(t')}_\cX}_{L^p_{\Omega}} 
        \leq C\max\cur{2,\sqrt{p}}T^{\frac{1}{2}}\norm{f}_{L^{\infty}_{\Omega}(L^{\infty}(0,T;\gamma(\cH;\cX)))}
    \end{equation*}
    for all $p \in [2,\infty)$.
    After using $2 \leq \sqrt{2}\sqrt{p}$, the result follows from \cref{lem:tail}.
\end{proof}

\section{Symmetry and linear stability}
\label{sec:linsym}
\subsection{Symmetry}
Let $\cX$ be a Banach space, and consider the evolution equation
\begin{equation}
    \label{eq:pde}
    \rd u = A u \rd t + F(u) \rd t,
\end{equation}
where $A$ is a closed (possibly unbounded) operator on $\cX$ with domain $\cD(A)$, and ${F \colon \cX \to \cX}$ is a nonlinear term.
Many (semilinear) PDEs can be written in this form.
A typical example is the case where $A$ is a variation on a Laplacian, and $F$ is a Nemytskii mapping on a function space with sufficient regularity.

Many interesting equations, especially physically motivated ones, are invariant under a group of continuous symmetries.
For PDEs formulated on $\bbR^d$, a common symmetry is invariance under translations and rotations, and this is the principal example we have in mind.
In this case, the symmetry group is the $d$-dimensional \emph{special Euclidean group}, denoted $\mathrm{SE}(d)$.
This group is generated by translations and rotations of $\bbR^d$, and is not commutative when $d \geq 2$.
However, there are many other examples of continuous symmetries, such as scalings, dilations, or more complicated gauge transformations.
This motivates the following assumption, which encodes such symmetries abstractly in the form of a Lie group.
For the reader unfamiliar with this subject, we refer to \cite{kirillov_introduction_2008}.
However, we do not require any theory beyond the basic notions of a Lie group, Lie algebra, and their representations.

\begin{assumption}[Symmetry of the equation]
    \label{ass:sympde}
    There exists a matrix Lie group $\grp$ and a group homomorphism $\Pi \colon \grp \to \cL(\cX)$ with the following properties:
    \begin{itemize}
        \item For $g \in \grp$ and $\phi \in \cD(A)$, we have $\Pi(g)\phi \in \cD(A)$ and
        \begin{equation}
            \label{eq:sympde}
            A\phi = \Pi(g) A \Pi(g^{-1})\phi, \qquad F(\phi) = \Pi(g)F(\Pi(g^{-1})\phi).
        \end{equation}
        \item For $\phi \in \cX$, the map $g \mapsto \Pi(g)\phi$ is continuous from $\grp$ to $\cX$. 
        \item There exists a constant $M$ such that $\norm{\Pi(g)}_{\cL(\cX)} \leq M$ for all $g \in \grp$.
    \end{itemize}
\end{assumption}
\begin{remark}
    \cref{ass:sympde} will guarantee that $\Pi(g)u(t)$ solves \cref{eq:pde} for any $g \in \grp$ whenever $u(t)$ solves \cref{eq:pde} and $u(t)$ is sufficiently regular.
\end{remark}
\begin{remark}
    By itself, \cref{ass:sympde} is trivial (as can be seen by taking $\grp$ as the trivial group).
    However, the requirement that $\grp$ is rich enough to capture all relevant symmetries of \cref{eq:pde} will be enforced by later assumptions.
\end{remark}

Although \cref{ass:sympde} is formulated in terms of the Lie group $\grp$, our following assumptions and results are formulated mostly in terms of its corresponding Lie algebra, which we denote by $\alg$.
By a slight abuse of notation, we will write $\exp$ or $e$ for the exponential map (which maps $\alg$ to $\grp$).
We also fix an arbitrary norm on $\alg$, to be used throughout the rest of the paper.
It is not important which norm we use, as all norms on $\alg$ are equivalent since $\alg$ is finite-dimensional.

A Lie group representation $\Pi$ typically gives rise to a Lie algebra representation $\pi$ via differentiation at the identity:
\begin{equation}
    \label{eq:piformula}
    \pi(Y) = \frac{\mathrm{d}}{\mathrm{d}t}\bigg\vert_{t = 0}\Pi(\exp(tY)), \quad Y \in \alg.
\end{equation}
However, if $\cX$ is infinite dimensional, $\pi(Y)$ is generally an unbounded operator, and we must take care that the limit inherent in \cref{eq:piformula} exists.
This motivates the following definition.
\begin{definition}
    \label{def:pi}
    Let $\grp$, $\Pi$ be as in \cref{ass:sympde}.
    For $Y \in \alg$, we define the unbounded operator
    \begin{equation}
        \pi(Y) \colon \phi \mapsto \lim_{t \to 0}t^{-1}(\Pi(e^{tY})\phi - \phi),
    \end{equation}
    where the limit is taken in the topology of $\cX$.
    The domain of $\pi(Y)$ consists of exactly the elements $\phi \in \cX$ for which this limit exists.
\end{definition}
\begin{remark}
    \cref{ass:sympde} directly implies that $t \mapsto \Pi(e^{tY})$ is a $C_0$-group on $\cX$ for any $Y \in \alg$.
    As $\pi(Y)$ can be seen to be its generator, it follows that $\pi(Y)$ is closed and densely defined.
\end{remark}

The main objects of study of this paper are \emph{symmetry solutions} to \cref{eq:pde}, by which we mean that the evolution of the solution is described purely by a symmetry of the equation.
Typical examples which we will keep in mind throughout are travelling waves, rotating waves, and stationary solutions.
It should be emphasized that existence of (nontrivial) symmetry solutions is generally a special property of a given equation.
A few references where such solutions are constructed are \cite{faria_nonmonotone_2006,alexander_topological_1990,arioli_existence_2015,kapitula_stability_1998,cohen_rotating_1978,hagan_spiral_1982,scheel_bifurcation_1998}.
Since our primary goal is to study nonlinear stability of stochastic perturbations of such solutions, we formulate their existence as an assumption.
\begin{assumption}[Existence of a regular symmetry solution]
    \label{ass:selfsimilarsolution}
    There exist $X \in \alg$ and $u^* \in \cD(A) \cap \cD(\pi(X))$ such that
    \begin{equation}
        \label{eq:uhat}
        \hat{u}(t) \coloneq \Pi(e^{tX})u^*, \quad t \geq 0,
    \end{equation}
    is a (strong) solution to \cref{eq:pde}.
    Furthermore, we have $u^* \in \cD(\pi(Y))$ for every $Y \in \alg$, and there exists a constant $C$ such that we have the estimate
    \begin{equation}
        \label{eq:taylorestimate}
        \norm{\Pi(e^Y)u^* - u^* - \pi(Y)u^*}_\cX \leq C \norm{Y}_{\alg}^2, \quad Y \in \alg.
    \end{equation}
\end{assumption}
\begin{remark}
    \cref{ass:selfsimilarsolution} covers stationary solutions, as can be seen by taking $X = 0$.
\end{remark}
As a concrete example, consider the case where $\cX = L^2(\bbR^d;\bbR)$, $\grp = \mathrm{SE}(d)$, and $X$ is the element in $\alg$ which generates translation by some vector $\vec{v} \in \bbR^d$.
Then we have $\Pi(\exp(tX))u^*(x) = u^*(x - t\vec{v})$, so \cref{ass:selfsimilarsolution} is satisfied if the equation has a travelling wave solution with wave velocity $\vec{v}$ and (sufficiently smooth) wave profile $u^*$.
It can also be seen that $\pi(X)$ is the directional derivative $-\partial_{\vec{v}}$.

\subsection{Dynamics in the comoving frame}
\label{subsec:dynamicscomoving}
We now transfer to a coordinate frame which is comoving with $\hat{u}(t)$.
Applying the transformation $\bar{u}(t) = \Pi(\exp(-tX))u(t)$, we get from \cref{eq:pde,ass:sympde}:
\begin{equation}
    \label{eq:pdecomoving}
    \rd \bar{u} = [A \bar{u} - \pi(X) \bar{u}] \rd t + F(\bar{u})\rd t.
\end{equation}
\cref{ass:selfsimilarsolution} then implies that $\bar{u}(t) \equiv u^*$ solves \cref{eq:pdecomoving} (in the strong sense). 
In fact, the reverse implication also holds, so we could replace \cref{ass:selfsimilarsolution} by the assumption that \cref{eq:pdecomoving} has a stationary solution $u^*$.

To study the dynamics of \cref{eq:pdecomoving} near $u^*$ (resp.\ \cref{eq:pde} near $\hat{u}(t)$), we will look at the linearization of \cref{eq:pdecomoving} around $u^*$.
The next assumption is sufficient for this linearization to be meaningful.
\begin{assumption}[Linearization]
    \label{ass:semigroup}
    The nonlinearity $F \colon \cX \to \cX$ is Fr\'echet differentiable at $u^*$.
    Furthermore, the operator
    \begin{equation}
        \label{eq:defLstar}
        \cL^* \colon \phi \mapsto A \phi - \pi(X)\phi + F'(u^*)\phi
    \end{equation}
    generates a bounded $C_0$-semigroup $\cur{S^*(t)}_{t \geq 0}$ on $\cX$.
\end{assumption}
Frow now on, we will use the terms $C_0$-semigroup and semigroup interchangeably.
We also emphasize that $\cL^*$ is \emph{not} an adjoint operator.
Instead, the $*$ superscript indicates that $\cL^*$ is associated with the comoving frame.
The same holds for the objects $S^*(t)$, $P^*_c$, and $P^*_s$, which will be introduced later.
\begin{remark}
    When the boundedness of $F'(u^*)$ is known, it follows by a perturbative argument that $\cL^*$ generates a semigroup whenever $A - \pi(X)$ generates a semigroup.
    As $A$ and $\pi(X)$ commute by \cref{ass:sympde}, verifying this can sometimes be reduced to checking that $A$ generates a semigroup, using the Trotter--Kato theorem.
\end{remark}
Naively, we might hope for the semigroup $S^*(t)$ to be exponentially stable, meaning that there exist constants $M,a > 0$ such that $\norm{S(t)}_{\cL(\cX)} \leq M e^{-at}$ for all $t \geq 0$.
By standard perturbative methods, this would imply that a solution which starts sufficiently close to $u^*$ will eventually converge to $u^*$.
However, it turns out that the presence of symmetries (in particular, non-triviality of the \emph{center space}) can pose a significant obstacle.

\begin{definition}
    \label{def:centerspace}
    The \emph{center space of $u^*$ with respect to $\alg$}, denoted by $\cV$, is defined as
    \begin{equation*}
        \cV \coloneq \cur{ \pi(Y)u^* : Y \in \alg}.
    \end{equation*}
\end{definition}
We emphasize that the center space is determined \emph{both} by the profile $u^*$ and the symmetry group $\grp$.
However, since we generally consider $u^*$ and $\grp$ to be fixed, we will simply speak of \emph{the} center space.
Note that $\pi(Y)u^*$ is well-defined by \cref{ass:selfsimilarsolution}, and the dimension of $\cV$ is at most that of $\alg$.
In particular, the center space is finite-dimensional, and thus closed in $\cX$.

It is important to note that the dimension of the center space can be strictly smaller than that of $\alg$.
The center space can in fact be trivial, even if $\alg$ is very rich.
In the case of PDEs on $\bbR^d$ with Euclidean symmetries, this occurs when $u^*$ is constant throughout space.

A guiding principle is that a nontrivial center space arises not due to symmetry directly, but due to a mismatch in symmetry between the profile $u^*$ and the equation.
The more symmetries present in the equation which are not shared by the profile $u^*$, the richer the center space.
Remarkably, the dynamics of $S^*(t)$ on the center space are entirely determined by the algebraic structure of $\alg$, as we will now demonstrate.

\begin{proposition}
    For any $Y \in \alg$ we have
    \begin{equation}
        \label{eq:Lstarcenter}
        \cL^* \pi(Y)u^* = \pi([Y,X])u^*.
    \end{equation}
    Consequently, $\cL^*$ can be restricted to a bounded operator on $\cV$.
\end{proposition}
\begin{proof}
    Let $Y \in \alg$.
    By \cref{ass:sympde} and the fact that $u^*$ is a stationary solution to \cref{eq:pdecomoving}, we have for $t \in \bbR$
    \begin{equation*}
        A(\Pi(\exp(tY))u^*) + F(\Pi(\exp(tY))u^*) - \Pi(\exp(tY))\pi(X)u^* = 0,
    \end{equation*}
    which we rewrite as
    \begin{equation}
    \label{eq:Lstarcenter1}
    \begin{aligned}
        A(\Pi(\exp(tY))u^*) &+ F(\Pi(\exp(tY))u^*) - \pi(X)\Pi(\exp(tY))u^* \\
        &= \Pi(\exp(tY))\pi(X)u^* - \pi(X)\Pi(\exp(tY))u^*.
    \end{aligned}
    \end{equation}
    We now claim that
    \begin{equation*}
        \frac{\mathrm{d}}{\mathrm{d}t}\bigg\vert_{t = 0} \bra[\big]{\Pi(\exp(tY))\pi(X)u^* - \pi(X)\Pi(\exp(tY))u^*} = \pi([Y,X])u^*.
    \end{equation*}
    If $\pi(X)u^* \in \cD(\pi(Y))$ and $\pi(Y)u^* \in \cD(\pi(X))$, this follows directly from \cref{def:pi}.
    If not, we approximate and use the fact that $\pi(X)$ and $\pi(Y)$ are closed and densely defined.
    Thus, the derivative at $t=0$ of the right-hand side of \cref{eq:Lstarcenter1} is well-defined and equal to $\pi([Y,X])u^*$.
    Since $\pi(X)$ and $A$ are closed and $F$ is differentiable at $u^*$, we may differentiate the left-hand side of $\cref{eq:Lstarcenter1}$ at $t = 0$ to get \cref{eq:Lstarcenter}.
    Thus, $\cL^*$ maps $\cV$ into $\cV$, and the boundedness follows since $\cV$ is finite-dimensional.
\end{proof}
The relation \eqref{eq:Lstarcenter} prompts us to define the (bounded) linear map
\begin{equation}
\begin{aligned}
    \label{eq:defL}
    L \colon \alg &\to \alg, \\
    Y &\mapsto [Y,X],
\end{aligned}
\end{equation}
so that \eqref{eq:Lstarcenter} can be formulated as
\begin{equation}
    \label{eq:LstarL}
    \cL^* \pi(Y)u^* = \pi(LY)u^*.
\end{equation}
Since the center space is finite-dimensional, we can define $e^{tL}$ and $e^{t\cL^*\vert_{\cV}}$ via the usual power series,
in which case $e^{t\cL^*\vert_{\cV}}$ coincides with $S^*(t)\vert_\cV$.
We now show that the relation \cref{eq:LstarL} lifts to a relation between $S^*(t)$ and $e^{tL}$.
This implies that the dynamics of $S^*(t)$ on $\cV$ are encoded entirely in the Lie bracket.
\begin{proposition}
    \label{prop:Sstarcenter}
    For $t \geq 0$ and $Y \in \alg$ we have
    \begin{equation}
        \label{eq:Sstarcenter}
        S^*(t)\pi(Y)u^* = \pi(e^{tL}Y) u^* = \pi(e^{-tX}Ye^{tX})u^*.
    \end{equation}
\end{proposition}
\begin{proof}
    Since the power series of $e^{t\cL^*\vert_{\cV}}$ and $e^{tL}$ both converge in the uniform topology, the first identity in \cref{eq:Sstarcenter} follows by iterating \cref{eq:LstarL}:
    \begin{equation*}
        S^*(t)\pi(Y)u^* = \sum_{n=0}^{\infty}\frac{(t\cL^*)^n}{n!}\pi(Y)u^* 
        = \pi\bra[\Big]{\sum_{n=0}^{\infty}\frac{(tL)^n}{n!}Y}u^*
        = \pi(e^{tL}Y)u^*.
    \end{equation*}
    We now observe that $L = -\ad_{X}$, where $\ad$ denotes the adjoint representation of $\alg$ on $\alg$.
    Thus, the second identity in \cref{eq:Sstarcenter} follows from the classical identity (see \cite[Lemma 3.14]{kirillov_introduction_2008})
    \begin{equation*}
        e^{-\ad_X} Y = e^{-X} Y e^{X}. \qedhere
    \end{equation*}
\end{proof}
As a final remark, we note that \cref{eq:Sstarcenter} immediately implies
\begin{equation}
    \label{eq:ustarinvariant}
    S^*(t)\pi(X)u^* = \pi(X)u^*, \quad t \geq 0,
\end{equation}
from which we see that $\pi(X)u^*$ is invariant under the dynamics of $S^*(t)$.
\subsection{Linear stability in the comoving frame}
\label{subsec:stabilitycomoving}
From \cref{prop:Sstarcenter}, it should now be clear (especially considering \cref{eq:ustarinvariant}) that we cannot expect exponential stability of $S^*(t)$ in general, unless the center space is trivial.
However, in many cases, a stability estimate can be recovered after `projecting out' the center space.
Essentially, this gives stability `modulo symmetry', and this is the reason why the concept of \emph{orbital stability} is needed.
Thus, our goal is now to find a space $\cW$ which is complementary to $\cV$ (in the sense that $\cX$ is the direct sum of $\cV$ and $\cW$) such that $S^*(t)$ leaves $\cW$ invariant and is exponentially stable on $\cW$.
We then call $\cW$ the \emph{stable space}.
If such a stable space exists, we can decompose the dynamics near $u^*$ into two parts: the dynamics on $\cV$ (which are purely determined by $\alg$), and the dynamics on $\cW$ (which are exponentially stable).
The next assumption guarantees that we have this decomposition.

\begin{assumption}[Decomposition]
    \label{ass:decomposition}
    There exist projections $P^*_c,P^*_s \in \cL(\cX)$ with the following properties:
    \begin{itemize}
        \item We have the decomposition $\eye = P^*_c + P^*_s$.
        \item The range of $P^*_c$ coincides with the center space $\cV$.
        \item The range of $P^*_s$ is stable under $S^*(t)$: there exist constants $M,a > 0$ such that
        \begin{equation}
            \label{eq:linstab}
            \norm{S^*(t)P^*_s}_{\cL(\cX)} \leq M e^{-at}, \quad t \geq 0.
        \end{equation}
    \end{itemize}
\end{assumption}
The stable space $\cW$ is given by the range of $P^*_s$.
The subscripts in $P^*_c$ and $P^*_s$ stand for \emph{center} and \emph{stable}, motivated by the fact that they project onto the center space and the stable space, respectively.
By \cref{def:centerspace} and finite-dimensionality of $\alg$, it follows that $P^*_c$ factorizes through a bounded linear map $\cP \colon \cX \to \alg$ as follows:
    \begin{equation}
        \label{eq:defP}
        P^*_c\phi = \pi(\cP \phi)u^*, \quad \phi \in \cX.
    \end{equation}
Typically, \cref{ass:decomposition} is verified using spectral methods and PDE techniques.
If $\cX$ is a Hilbert space, a sufficient condition on the spectrum of $\cL^*$ may be formulated using the Gearhart--Pr\"uss theorem.
In this case, it suffices that the spectrum is of the form
\begin{equation}
    \label{eq:negativespectrum}
    \sigma(\cL^*_\bbC) \subset \sigma(L_\bbC) \cup \cur{ z \in \bbC : \Real z \leq -b}
\end{equation}
for some $b > 0$, and that the eigenspaces corresponding to $\sigma(L_\bbC)$ are spanned by $\cV_\bbC$ 
(the condition \cref{eq:negativespectrum} is formulated with respect to the complexification of $L$ and $\cL^*$, since these operators are typically not selfadjoint).
The assumptions on the spectrum are then verified by analyzing the linearized PDE directly.

We note that versions of \cref{ass:decomposition} and the spectral condition \cref{eq:negativespectrum} are common in the literature on stability of deterministic and stochastic patterns.
For a (non-exhaustive) list of examples where this is explicitly assumed or proven, one can consider \cite[Lemma 3.1]{hamster_stability_2020}, \cite[Assumption 3.1]{inglis_general_2016}, \cite[Assumption 3.5]{kuehn_stochastic_2022}, \cite[Assumption 2.6]{maclaurin_phase_2023}, \cite[Proposition 2.8]{eichinger_multiscale_2022}, \cite{alexander_topological_1990}.
Moreover, \cref{ass:decomposition} can be obtained as a corollary of many of the other cited linear stability results.

We also emphasize that \cref{eq:negativespectrum} allows $\cL^*_{\bbC}$ to have (point) spectrum which is located on the imaginary axis and not at the origin.
This actually occurs for the rotating wave, where we have $\sigma(L_\bbC) = \cur{0,\pm i \omega}$.
One can compare \cref{eq:negativespectrum} with \cite[Assumption 2.6]{maclaurin_phase_2023}, which explicitly forbids this situation and only treats commutative symmetry.
In the commutative case, $L \equiv 0$ and all spectrum on the imaginary axis must be at the origin.

\subsection{Return to the stationary frame}
We ultimately want to solve a stochastic version of \cref{eq:pde} in the stationary frame.
To do this, we first need to formulate a solution concept.
Motivated by \cref{ass:semigroup}, we first rewrite \cref{eq:pdecomoving} as
\begin{align*}
    \rd \bar{u} &= [A - \pi(X) + F'(u^*)]\bar{u} \rd t + [F(\bar{u}) - F'(u^*)\bar{u}] \rd t \\
        \overset{\eqref{eq:defLstar}}&{=} \cL^* \bar{u} \rd t + [F(\bar{u}) - F'(u^*)\bar{u}] \rd t.
\end{align*}
By Duhamel's principle (variation of parameters), a solution to \cref{eq:pdecomoving} with initial value $u_0 \in \cX$ should then satisfy
\begin{equation*}
    \bar{u}(t) = S^*(t)u_0 + \int_0^t S^*(t-t')\bra[\big]{F(\bar{u}(t')) - F'(u^*)\bar{u}(t')} \rd t'.
\end{equation*}
Undoing the transformation $\bar{u}(t) = \Pi(e^{-tX})u(t)$ and using the symmetries of $F$ from \cref{ass:sympde}, it follows that a solution to \cref{eq:pde} with initial value $u_0$ should satisfy the \emph{mild solution formula}
\begin{equation*}
    u(t) = \Pi(e^{tX})S^*(t)u_0 + \int_0^t \Pi(e^{tX})S^*(t-t')\Pi(e^{-t'X})\bra[\big]{F(u(t')) - F'(\hat{u}(t'))u(t')} \rd t'.
\end{equation*}
We now define for $0 \leq t' \leq t$ the following bounded linear operators ($P_c(t)$ and $P_s(t)$ are intended for later purposes):
\begin{subequations}
\label{eq:stationaryoperators}
\begin{align}
    \label{eq:defStt}
    S(t,t') &\coloneq \Pi(e^{tX})S^*(t-t')\Pi(e^{-t'X}), \\
    \label{eq:defPc}
    P_c(t) &\coloneq \Pi(e^{tX})P^*_c\Pi(e^{-tX}), \\
    \label{eq:defPs}
    P_s(t) &\coloneq \Pi(e^{tX})P^*_s\Pi(e^{-tX}),
\end{align}
\end{subequations}
so that the solution formula simplifies to its final version
\begin{equation}
    \label{eq:milddet}
    u(t) = S(t,0)u_0 + \int_0^t S(t,t')\bra[\big]{F(u(t')) - F'(\hat{u}(t'))u(t')} \rd t'.
\end{equation}
The following proposition translates the implications of \cref{ass:sympde,ass:selfsimilarsolution,ass:semigroup,ass:decomposition} back to the stationary frame.
\begin{proposition}
    \label{prop:stationarysummary}
    The following statements hold:
    \begin{itemize}
        \item $\cur{S(t,t')}_{0 \leq t' \leq t}$ is a $C_0$-evolution family on $\cX$.
        \item For $t \geq 0$, the bounded operators $P_c(t)$ and $P_s(t)$ are projections, and we have the decomposition $\eye = P_c(t) + P_s(t)$.
        \item For $0 \leq t' \leq t$ and $\phi \in \cX$ we have the identities
        \begin{subequations}
            \label{eq:Sttidentities}
        \begin{align}
            \label{eq:SttPc}
            S(t,t')P_c(t')\phi &= \Pi(e^{tX})\pi(e^{(t-t')L}\cP\Pi(e^{-t'X})\phi)u^*, \\
            \label{eq:Sttcenter}
            S(t,t')\pi(Y)\hat{u}(t') &= \Pi(e^{tX})\pi(e^{tL}Y)u^*.
        \end{align}
        \end{subequations}
        \item There exist constants $M_1,M_2,M_3,a > 0$ such that we have the estimates
            \begin{subequations}
                \label{eq:Sttestimates}
                \begin{align}
                    \label{eq:SttestimatePi}
                    \norm{\Pi(g)}_{\cL(\cX)} &\leq M_1, \\
                    \label{eq:SttestimateStt}
                    \norm{S(t,t')}_{\cL(\cX)} &\leq M_2, \\
                    \label{eq:SttestimateSttPs}
                    \norm{S(t,t')P_s(t')}_{\cL(\cX)} &\leq M_3e^{-a(t-t')},
                \end{align}
            \end{subequations}
            for all $0 \leq t' \leq t$ and $g \in \grp$.
    \end{itemize}
\end{proposition}
Although \cref{eq:SttestimatePi} was already stated in \cref{ass:sympde}, we include it again here to have all the relevant constants in one place for later use.
\begin{proof}
    Most statements follow straightforwardly from \cref{eq:stationaryoperators} and the corresponding properties of $S^*(t)$, $P^*_c$, and $P^*_s$, so we only prove \cref{eq:Sttidentities}.
    From \cref{ass:decomposition,prop:Sstarcenter} we see that
    \begin{align*}
        S(t,t')P_c(t')\phi &\overset{\eqref{eq:stationaryoperators}}{=} \Pi(e^{tX})S^*(t-t')P^*_c\Pi(e^{-t'X})\phi \\
        &\overset{\eqref{eq:defP}}{=} \Pi(e^{tX})S^*(t-t')\pi(\cP\Pi(e^{-t'X}\phi))u^* \\
        &\overset{\eqref{eq:Sstarcenter}}{=}\Pi(e^{tX})\pi(e^{(t-t')L}\cP\Pi(e^{-t'X}\phi))u^*,
    \end{align*}
    as well as
    \begin{align*}
        S(t,t')\pi(Y)\hat{u}(t') \overset{\eqref{eq:uhat},\eqref{eq:defStt}}&{=} \Pi(e^{tX})S^*(t-t')\Pi(e^{-t'X})\pi(Y)\Pi(e^{t'X})u^* \\
        \overset{\eqref{eq:Sstarcenter}}&{=} \Pi(e^{tX})S^*(t-t')\pi(e^{t'L}Y)u^* \\
        \overset{\eqref{eq:Sstarcenter}}&{=} \Pi(e^{tX})\pi(e^{(t-t')L}e^{t'L}Y)u^*,
    \end{align*}
    which implies \cref{eq:Sttcenter}.
\end{proof}

\section{Nonlinear stability}
\label{sec:nonlinstab}

\subsection{Orbital stability and the predicted phase}
\label{subsec:orbstab}
We now study the stability of $\hat{u}(t)$ in the full nonlinear equation \cref{eq:pde}.
As the discussion in the previous section suggests, we generally cannot expect solutions starting close to $\hat{u}(0) = u^*$ to converge to $\hat{u}(t) = \Pi(e^{tX})$.
Instead, we will find that a solution which starts close to $u^*$ will, after some positive time $t$, be close to $\Pi(\gamma)u^*$ for a $\gamma \in \grp$ which is different from $e^{tX}$.
Hence, it is the \emph{center manifold} $\cC$, given by the group orbit $\cC \coloneq \cur{\Pi(g)u^* : g \in \grp}$ which is stable, instead of the solution $\hat{u}(t)$.
This is precisely what we mean by \emph{orbital stability}.
We aim to answer the following questions:
\begin{enumerate}
    \item How do we prove stability of the center manifold?
    \item Which symmetries should be included in $\grp$ for the center manifold to be stable?
    \item How can we compute the \emph{phase} $\gamma \in \grp$ from the initial perturbation $u_0 - u^*$?
\end{enumerate}
These questions are strongly interlinked.
To prove stability, we will need to know the right symmetry group and the correct phase in advance.
Conversely, a stability proof ensures that the phase shift and the symmetry group used are the correct ones.

To answer the second and third question, we look to the linear theory developed in the previous section.
We consider a mild solution $u(t)$ to \cref{eq:pde} with initial condition $u(0) = u^* + v_0$, where $v_0 = \cO(\eps)$ with $\eps \ll 1$.
Since the evolution family $S(t,t')$ arises from linearization around $\hat{u}(t)$, we can heuristically expect that
\begin{equation}
    \label{eq:uapprox}
    u(t) = \hat{u}(t) + S(t,0)v_0 + \cO(\eps^2).
\end{equation}
Using the decomposition $\eye = P_c(0) + P_s(0)$ from \cref{prop:stationarysummary} together with \cref{eq:SttPc}, we find
\begin{equation*}
    u(t) = \Pi(e^{tX})u^* + \Pi(e^{tX})\pi(e^{tL}\cP v_0)u^* +  S(t,0)P_s(0)v_0 + \cO(\eps^2).
\end{equation*}
In the first two terms on the right-hand side, we recognize an abstract Taylor expansion of $\Pi(e^{tX}\exp(e^{tL}\cP v_0))u^*$ (c.f.\ \cref{eq:taylorestimate}).
The third term is exponentially decaying with time by \cref{eq:SttestimateSttPs}.
Hence, we have
\begin{equation}
    \label{eq:orbstabintuitive}
    u(t) = \Pi\bra[\big]{\exp(tX)\exp(e^{tL}\cP v_0)}u^* + \cO(\eps e^{-at} + \eps^2),
\end{equation}
which shows that the difference between $u(t)$ and $\Pi\bra[\big]{\exp(tX)\exp(e^{tL}\cP v_0)}u^*$ is expected to be smaller than the difference between $u(0)$ and $u^*$, when $t$ is large enough and $\eps$ is sufficiently small.
Since the exponential decay from \cref{eq:SttestimateSttPs} originates from \cref{ass:decomposition}, it is now clear that we must include in $\grp$ enough symmetries for this assumption to hold.
Recalling the discussion in \cref{sec:linsym}, the orbital degrees of freedom thus originate from exactly those symmetries which are present in the equation \cref{eq:pde} but not in the profile $u^*$.

Motivated by \cref{eq:orbstabintuitive}, we now introduce the \emph{predicted phase at time $t$}, denoted $\gamma_t \colon \cX \to \grp$, as
\begin{equation}
    \label{eq:predictedphase}
    \gamma_t \colon v_0 \mapsto \exp(tX)\exp(e^{tL}\cP v_0).
\end{equation}
This is the notion of phase which we will employ to show orbital stability.
In \cref{eq:predictedphase}, $v_0$ should be interpreted as the deviation from the profile $u^*$ at time $t = 0$.
When $v_0$ is clear from the context, we will simply write $\gamma_t$ instead of $\gamma_t(v_0)$ to unburden the notation.
\begin{remark}
    If $v_0 = 0$, then $\gamma_t = \exp(tX)$. This corresponds to the fact that $\hat{u}(t) = \Pi(e^{tX})u^*$ is an exact solution of \cref{eq:pde}. Hence, the predicted phase matches the exact phase in the absence of a perturbation.
    Moreover, this suggests that we can interpret $\exp(e^{tL}\cP v_0)$ in \cref{eq:predictedphase} as the \emph{phase correction} relative to the unperturbed evolution.
\end{remark}
\begin{remark}
    It is seen from \cref{eq:predictedphase,eq:defL} that any noncommutativity of $\alg$ (encoded in $L$), enters the predicted phase in a nontrivial way.
\end{remark}
A compelling feature of the predicted phase is that it is determined directly from the initial perturbation $v_0$.
Moreover, the right-hand side of \cref{eq:predictedphase} can often be computed explicitly.
In most concrete situations, the symmetry $\exp(tX)$ is explicitly known, so that the matrices $L$ and $e^{tL}$ can be calculated by hand.
Furthermore, $\cP$ can often be described in terms of the eigenfunctions of the formal adjoint of $\cL^*$.
We believe this feature sets our method apart from the various phase tracking methods in the literature.
For explicit computations of \cref{eq:predictedphase} in concrete situations, we refer ahead to \cref{rem:fhnphase,subsec:rotwave}.

As a second distinguishing feature, we note that the predicted phase, as well as all of our assumptions so far, is formulated without any reference to a Hilbertian structure of $\cX$.
In fact, up to this point, $\cX$ can be an arbitrary Banach space.
We believe this flexibility to be advantageous, especially if one wants to treat measure-valued equations, or (S)PDEs which are well-posed in an $L^p$-setting with $p \neq 2$ (see e.g.\ \cite{agresti_nonlinear_2022a,agresti_reactiondiffusion_2023,agresti_nonlinear_2022}).

\subsection{Stability in the deterministic setting}
Our next step is to use the predicted phase to make \cref{eq:orbstabintuitive} rigorous.
This will answer the primary question on how to show stability in the deterministic setting,
and also demonstrates that the phase shift predicted by \cref{eq:predictedphase} is indeed accurate.

We begin by formulating a lemma which shows how \cref{eq:orbstabintuitive} follows from \cref{eq:uapprox}, and which will be used for the coming orbital stability proofs.
The way in which we have written \cref{lem:nonlinstep1} is suggestive of how we will apply it.
\begin{lemma}
    \label{lem:nonlinstep1}
    Let $T > 0$, $v_0 \in \cX$ and $f \in C([0,T];\cX)$.
    Let $\gamma_t$ be given by \cref{eq:predictedphase}, and set
    \begin{equation*}
        u(t) \coloneq \hat{u}(t) + S(t,0)v_0 + f(t), \quad t \in [0,T].
    \end{equation*}
    Then we have the estimate
    \begin{equation*}
        \norm{u(t) - \Pi(\gamma_t)u^*}_\cX 
        \leq M_3e^{-at}\norm{v_0}_\cX + C M_1 K_T^2 \norm{v_0}^2_\cX + \norm{f(t)}_\cX, \quad t \in [0,T],
    \end{equation*}
    where $C$ is the constant from \cref{eq:taylorestimate}, 
    $M_1,M_3,a$ are the constants from \cref{eq:Sttestimates}, 
    and $K_T = \sup_{t \in [0,T]}\norm{e^{tL}\cP}_{\cL(\cX;\alg)}$.
\end{lemma}
\begin{proof}
    For brevity, we write $Y_t = e^{tL}\cP v_0$.
    By the decomposition $\eye = P_c(0) + P_s(0)$ and \cref{eq:SttPc}, we have
    \begin{align*}
        S(t,0)v_0 &= S(t,0)P_c(0) v_0 + S(t,0)P_s(0) v_0 \\
        &= \Pi(e^{tX})\pi(Y_t)u^* + S(t,0)P_s(0)v_0
    \end{align*}
    for $t \in [0,T]$.
    Hence, by the triangle inequality, we get
    \begin{align*}
        \norm{u(t) - \Pi(\gamma_t) u^*}_\cX &\leq \norm{S(t,0)P_s(0)v_0 + f(t)}_\cX \\
        &\quad+ \norm{\Pi(\gamma_t)u^* - \hat{u}(t) - \Pi(e^{tX})\pi(Y_t)u^*}_\cX,
    \end{align*}
    so it suffices to estimate these two terms.
    For the first term, we observe using \cref{eq:SttestimateSttPs}:
    \begin{equation*}
        \norm{S(t,0)P_s(0)v_0 + f(t)}_\cX \leq M_3 e^{-at}\norm{v_0}_\cX + \norm{f(t)}_\cX.
    \end{equation*}
    For the second term, we use \cref{eq:predictedphase} to see $\gamma_t = e^{tX}e^{Y_t}$, which together with \cref{eq:taylorestimate,eq:SttestimatePi} gives:
    \begin{align*}
        \norm{\Pi(\gamma_t)&u^* - \hat{u}(t) - \Pi(e^{tX})\pi(Y_t)u^*}_\cX
        \leq M_1\norm{\Pi(e^{Y_t})u^* - u^* - \pi(Y_t)u^*}_\cX \\
        &\leq C M_1\norm{Y_t}^2_\alg \leq C M_1 \norm{e^{tL}\cP}_{\cL(\cX;\alg)}^2 \norm{v_0}^2_\cX. \qedhere
    \end{align*}
\end{proof}

Before we state the nonlinear stability result, we need to assume some regularity on the nonlinearity $F$.
Recall that we have previously assumed that $F$ is Fr\'echet differentiable at $u^*$ (see \cref{ass:semigroup}).
\begin{assumption}[Regularity of $F$ near $u^*$]
    \label{ass:Fregular}
    For every $R > 0$, there exists a constant $C$ such that the estimate
        \begin{equation}
            \label{eq:Festimates}
            \norm{F(v) - F(u^*) - F'(u^*)[v - u^*]}_\cX \leq C \norm{v - u^*}_\cX^2
        \end{equation}
        holds for all $v \in \cX$ satisfying $\norm{v - u^*} \leq R$.
\end{assumption}
\begin{remark}
    From \cref{ass:sympde}, it follows that we can replace $u^*$ in \cref{eq:Festimates} by $\Pi(g)u^*$ for any $g \in \grp$.
    In particular, \cref{eq:Festimates} holds with $u^*$ replaced by $\hat{u}(t)$ for any $t \geq 0$.
\end{remark}

We can now state the deterministic orbital stability result.
\cref{thm:detstab} confirms the heuristic discussion in \cref{subsec:orbstab}, and rigorously shows that \cref{eq:predictedphase} gives an accurate prediction of the phase.
The proof of \cref{thm:detstab} is contained in \cref{subsec:proofdetstab}.
\begin{theorem}[Orbital stability]
    \label{thm:detstab}
    Let $u(t)$ be a solution to \cref{eq:milddet} with initial condition $u_0 = u^* + v_0$, and let $\gamma_t$ be given by \cref{eq:predictedphase}.
    For every $T > 0$, $\delta > 0$, there exists a constant $\eps > 0$ such that we have the estimate
    \begin{equation}
        \label{eq:detorbstab}
        \norm{u(t) - \Pi(\gamma_t)u^*}_\cX \leq (M_3 e^{-at} + \delta)\norm{v_0}_\cX, \quad t \in [0,T],
    \end{equation}
    whenever $\norm{v_0}_\cX \leq \eps$ ($M_3$ and $a$ are the constants from \cref{eq:SttestimateSttPs}).
\end{theorem}
Interestingly, the estimate \cref{eq:detorbstab} shows that the solution gets closer to the center manifold only for $t \geq a^{-1}\log(M_3)$.
If $S(t,t')P_s(t')$ is not immediately contractive, the solution might move away from the center manifold between $t=0$ and $t=a^{-1}\log(M_3)$.

Using the symmetries of \cref{eq:pde} encoded in \cref{ass:sympde}, it is a corollary of \cref{thm:detstab} that the center manifold is exponentially attracting for the deterministic equation.
We elaborate on this point in the next section (see \cref{thm:stochstablong}), where we treat stochastic perturbations.

\subsection{Stochastic perturbations}
\label{subsec:stochpert}
In this section, we use the predicted phase \cref{eq:predictedphase} to show stability of stochastic perturbations of \cref{eq:pde} on long timescales.
We believe the conciseness of the proof demonstrates the strength and elegance of the phase prediction function.

For $\sigma > 0$, we consider the following SPDE:
\begin{equation}
\begin{aligned}
    \label{eq:spde}
    \rd u(t) &= [A u(t) + F(u(t)) + \sigma^2 H(t,u(t))] \rd t + \sigma G(t,u(t)) \rd W(t), \\
        u(0) &= u^* + v_0,
\end{aligned}
\end{equation}
where $A$, $F$ are as before, $\cH$ is a separable Hilbert space, $W(t)$ is an $\cH$-cylindrical Wiener process, $H$ takes values in $\cX$, and $v_0$ is an $\cF_0$-measurable, $\cX$-valued random variable.
Since we are working in an abstract Banach space, we let $G$ take values in the space of $\gamma$-radonifying operators from $\cH$ to $\cX$, denoted $\gamma(\cH;\cX)$ (see \cref{subsec:notation}).
For more details and abstract results about such spaces, we refer to \cite[Chapter 9]{hytonen_analysis_2016}.
In \cref{subsec:fhn}, we show in a concrete example what kind of noise is allowed under this condition.
We allow $G$ and $H$ to depend on $\omega$, but generally suppress this dependence in the notation.
As discussed in \cref{subsec:tail}, we assume from now on that $\cX$ is $2$-smooth.

The term $\sigma G \rd W$ models noise present in the system, the amplitude of which is controlled by $\sigma$.
We interpret \cref{eq:spde} in the It\^o sense, and include the drift term $\sigma^2 H\rd t$ in \cref{eq:spde} to account for a possible Stratonovich correction.
However, we do not impose any relation between $G$ and $H$.
Hence, this setup allows for arbitrary deterministic perturbations which are of second order in $\sigma$.

We require some regularity assumptions on $G$ and $H$ for \cref{eq:spde} to be well-posed.
Since the aim of this paper is not to treat well-posedness of stochastic PDEs, we make do with the following relatively simple local Lipschitz assumptions.
For any concrete equation, more appropriate function spaces and nonlinear estimates can be formulated.
As an example of such a treatment, we refer to \cite{gnann_solitary_2024}, where stability of a solitary wave in a modified nonlinear Schr\"odinger equation is shown using dispersive estimates.
\begin{assumption}[Regularity of $G$ and $H$]
\label{ass:GHregular}
    For $f \in \cX$, the processes $(\omega,t) \mapsto G(\omega,t,f)$ and $(\omega,t) \mapsto H(\omega,t,f)$ are progressively measurable.
    For every $R > 0$, there exist a constant $C$ such that the estimates
    \begin{subequations}
    \begin{align}
        \norm{H(t,f)}_\cX &\leq C, \\
        \label{eq:Gbounded}
        \norm{G(t,f)}_{\gamma(\cH;\cX)} &\leq C, \\
        \norm{H(t,f) - H(t,g)}_\cX &\leq C \norm{f - g}_\cX, \\     
        \label{eq:Glipschitz}   
        \norm{G(t,f) - G(t,g)}_{\gamma(\cH;\cX)} &\leq C \norm{f - g}_\cX.
    \end{align}
    \end{subequations}
    are valid for all $\omega \in \Omega$, $t \in [0,\infty)$ and $f,g \in \cX$ which satisfy $\max\cur{\norm{f}_\cX, \norm{g}_\cX} \leq R$.
\end{assumption}

It is well-known that stochastic evolution equations with globally Lipschitz coefficients and linear growth have unique local solutions.
The following basic well-posedness result can be obtained from \cite[Theorem 7.2]{daprato_stochastic_1992a} by a straightforward localization argument.
\begin{theorem}[Local well-posedness]
    \label{thm:localwellposed}
    There exist unique (up to indistinguishability) random variables $(\tau^*,\cur{u(t)}_{t \in [0,\tau^*)})$ such that:
    \begin{itemize}
        \item $\tau^*$ is a stopping time.
        \item $u$ is adapted and continuous on $[0,\tau^*)$, with values in $\cX$.
        \item For all $t \in [0,\tau^*)$, $u(t)$ solves the following mild formulation to \cref{eq:spde}:
        \begin{equation}
        \label{eq:spdemild}
        \begin{aligned}
            u(t) &= S(t,0)(u^* + v_0) + \int_0^t S(t,t')\sigma G(t',u(t')) \rd W(t') \\
            &\quad+ \int_0^t S(t,t')\bra[\big]{F(u(t')) - F'(\hat{u}(t'))u(t') + \sigma^2H(t',u(t')) } \rd t'.
        \end{aligned}
        \end{equation}
        \item If $\tau^* < \infty$, then $\lim_{t \nearrow \tau^*} \norm{u(t)}_\cX = \infty$.
    \end{itemize}
\end{theorem}

From now on, we will write $u$ and $\tau^*$ for the solution and stopping time obtained from \cref{thm:localwellposed}.
Since the solution to \cref{eq:spde} with $\sigma = 0$ and $v_0 \equiv 0$ is given by $u(t) \equiv \hat{u}(t)$, we propose the following asymptotic expansion for $u$:
\begin{equation*}
    u(t) = \hat{u}(t) + v(t) + z(t),
\end{equation*}
where $v$ is $\cO(\sigma)$ and $z$ is $\cO(\sigma^2)$.
Substituting this ansatz into \cref{eq:spde} and grouping the terms based on their order in $\sigma$, we get
\begin{align*}
    \label{eq:lin}
    \rd u(t) &= \bra[\big]{A \hat{u}(t) + F(\hat{u}(t))} \rd t \\
    &\quad+  \bra[\big]{[A v(t) + F'(\hat{u}(t))v(t)] \rd t + \sigma G(t,\hat{u}(t)) \rd W(t)} \\
    &\qquad+ \bra[\big]{A z(t) + F'(\hat{u}(t))z(t)} \rd t \\
    &\qquad+ \bra[\big]{F(u(t)) - F(\hat{u}(t)) - F'(\hat{u})[u(t) - \hat{u}(t)]}\rd t \\
    &\qquad+ \sigma^2 H(t,u(t)) \rd t \\
    &\qquad+ \sigma \bra[\big]{G(t,u(t)) - G(t,\hat{u}(t))} \rd W(t).
\end{align*}
Note that $\rd \hat{u}(t) = \bra[\big]{A \hat{u}(t) + F(\hat{u}(t))}\rd t$ is satisfied, and $\hat{u}(0) = u^*$.
Hence, for $v$ we obtain the equation
\begin{subequations}
\label{eq:dvdz}
\begin{equation}
\begin{aligned}
    \label{eq:dv}
    \rd  v(t) &= [A + F'(\hat{u}(t))] v(t)\rd t + \sigma G(t,\hat{u}(t)) \rd W(t), \\
    v(0) &= v_0,
\end{aligned}
\end{equation}
and for $z$ we get
\begin{equation}
\label{eq:dz}
\begin{aligned}
    \rd z(t) &= [A + F'(\hat{u}(t))] z(t)\rd t + \bra[\big]{F(u(t)) - F(\hat{u}(t)) - F'(\hat{u}(t))[u(t) - \hat{u}(t)]}\rd t \\
    &\quad+ \sigma^2 H(t,u(t)) \rd t + \sigma\bra[\big]{G(t,u(t)) - G(t,\hat{u}(t))} \rd W(t), \\
    z(0) &= 0.
 \end{aligned}
\end{equation}
\end{subequations}
Since we already know from \cref{thm:localwellposed} that there is a unique solution to \cref{eq:spde}, it is not difficult to obtain solutions for \cref{eq:dv,eq:dz}.
From \cref{eq:dv}, we should expect to obtain a solution formula for $v$ which does not involve $u$ or $z$.
This is indeed possible.

\begin{proposition}[Asymptotic expansion]
    \label{prop:asym}
    Equation \cref{eq:dv} has a unique mild solution, explicitly given by
    \begin{equation}
        \label{eq:defv}
        v(t) = S(t,0)v_0 + \sigma \int_0^t S(t,t')G(t',\hat{u}(t'))\rd W(t'), \qquad t \in [0,\infty).
    \end{equation}
    With this solution, we have the following first-order asymptotic expansion:
    \begin{equation}
        \label{eq:expansion}
        u(t) = \hat{u}(t) + v(t) + z(t), \qquad t \in [0,\tau^*),
    \end{equation}
    where $z$ (as defined by the right-hand side of \cref{eq:expansion}) is the mild solution to \cref{eq:dz}.
\end{proposition}
The proof consists of substituting \cref{eq:expansion} into \cref{eq:spdemild}, and rearranging the terms in the same way that \cref{eq:dvdz} was derived.
Since this amounts to repeating the calculation we just performed, we do not write this out.

With the asymptotic expansion established, we can formulate the first stochastic stability result.
\cref{thm:stochstabshort} can be seen as a stochastic generalization of \cref{thm:detstab}.
The proof is contained in \cref{subsec:proofstochstabshort}.

\begin{theorem}[Short-term exponential stability]
    \label{thm:stochstabshort}
    Let $\gamma_t$ be given by \cref{eq:predictedphase}.
    For every $T > 0$, $\alpha > 1$, there exist constants $c,\eps'$ such that we have the estimate
    \begin{equation}
        \label{eq:stochstabshort}
        \PP[\Big]{\sup_{t \in [0,T \wedge \tau^*)} \norm{u(t) - \Pi(\gamma_t)u^*}_\cX - M_3\alpha e^{-at} \eps \geq \eps,\,  \norm{v_0}_\cX \leq \alpha \eps} \leq 6\exp(-c \eps^2 \sigma^{-2}),
    \end{equation}
    for all $\eps,\sigma$ satisfying $0 < \sigma \leq \eps \leq \eps'$, and any initial condition $v_0$ ($M_3$, $a$ are the constants from \cref{eq:SttestimateSttPs}).
\end{theorem}
Theorem \ref{thm:stochstabshort} does not directly imply stability on long timescales.
The reason is that $\eps'$ and $c$ both depend on $T$, and from the proof it can be seen that $\eps' \sim T^{-1}$.
Thus, a direct application of the theorem can only give stability on a timescale $T \sim \sigma^{-1}$.
However, if we choose $\alpha = 2$ and $T  = a^{-1}\log(M_3 \alpha)$, then \cref{eq:stochstabshort} implies
\begin{equation*}
    \PP[\big]{ \norm{u(T) - \Pi(\gamma_T)u^*}_\cX \geq 2 \eps,\,  \norm{v_0}_\cX \leq 2 \eps} \leq 6\exp(-c \eps^2 \sigma^{-2}).
\end{equation*}
Hence, if the solution at time $t = 0$ is close to the center manifold, then the solution at time $t = T$ will be equally close (with high probability).
Using the symmetries of the equation, it is possible to show a similar estimate on the interval $[T,2T]$, and so on.
Combining these estimates then results in the following long-time stability result.
For simplicity, we use the initial condition $v_0 \equiv 0$, but a similar result holds when $\norm{v_0} \leq \eps$.
The proof of \cref{thm:stochstablong} is contained in \cref{subsec:proofstochstablong}.

\begin{theorem}[Long-term stability]
    \label{thm:stochstablong}
    Set $v_0 \equiv 0$. 
    There exist constants $c,\eps'> 0$ such that we have the estimate
    \begin{equation}
        \label{eq:stochstablong}
        \PP[\Big]{ \sup_{t \in [0,T \wedge \tau^*)} \inf_{\gamma \in \grp} \norm{u(t) - \Pi(\gamma)u^*}_\cX \geq \eps}
        \leq 12\, T \exp(-c \eps^2 \sigma^{-2}),
    \end{equation}
    for all $T > 0$ and all $\eps,\sigma$ satisfying $0 < \sigma \leq \eps \leq \eps'$.
\end{theorem}
Notice that unlike in \cref{thm:stochstabshort}, the constants $c,\eps'$ in \cref{thm:stochstablong} do not depend on $T$.
From the proof of \cref{thm:stochstablong}, it can also be seen that the phase at time $nT$ (where $T$ is a fixed, large enough constant and $n \in \bbN$) in \cref{eq:stochstablong} is given by the recursive relation
\begin{equation}
    \label{eq:simplephaselong}
    \gamma((n+1)T) = \gamma(nT) e^{TX} \exp(e^{TL}\cP[\Pi(\gamma(nT)^{-1})u(nT) - u^*]), \quad n \in \bbN,
\end{equation}
where $\gamma(0)$ is the identity element of $\grp$.
We emphasize the remarkable phenomenon that the phase at time $(n+1)T$ is already determined by the solution at time $nT$.
This shows that it is generally possible to `look into the future', and accurately predict where the pattern will be after time $T+t$, solely based on the profile at time $T$.

Finally we note that, although \cref{eq:simplephaselong} seems quite explicit, it does contain the solution $u$ on the right-hand side.
Hence, to compute $\gamma(nT)$ for large $n$, explicit information about the solution is required.
However, this feature is shared by competing definitions of the phase, and we are not aware of any way to exactly compute the motion of moving patterns without explicit knowledge of the solution.

\section{Examples}
\label{sec:examples}
In this section, we revisit and improve on some examples from the existing literature.
The first example was chosen to highlight the flexibility of working in Banach spaces, and the second example demonstrates how the phase may be concretely computed in a noncommutative scenario.

\subsection{Travelling pulse in a FitzHugh--Nagumo equation}
\label{subsec:fhn}
We consider the following stochastic fully diffusive FitzHugh--Nagumo equation:
\begin{equation}
\label{eq:fhn}
\begin{aligned}
    \partial_t u(t,x) &= \partial_x^2 u(t,x)  + f(u(t,x)) - v(t,x), \\
    \partial_t v(t,x) &= \delta \partial_x^2 v(t,x) + \eps(u(t,x) - \gamma v(t,x)),
\end{aligned}
\end{equation}
with $(t,x) \in \bbR^+ \times \bbR$.
Here, $\delta,\eps,\gamma$ are positive constants, and $f(u) = u(u-a)(1-u)$ for some $a \in (0,\tfrac{1}{2})$.
Existence of travelling pulse solutions to \cref{eq:fhn} with $\delta = 0$ was shown in \cite{conley_application_1984,gardner_existence_1983}, and stability was shown in \cite{jones_stability_1984}.
Stability of stochastic perturbations (with $\delta = 0$) is shown in \cite{eichinger_multiscale_2022}, using the phase-lag method introduced by Kr\"uger and Stannat \cite{kruger_front_2014}.

Although our assumptions allow us to take $\delta = 0$, we prefer to treat the case $\delta > 0$.
Stability of the fast pulse in this situation was treated in \cite{alexander_topological_1990}.
There, it is shown there exist $\delta,\eps,\gamma,c>0$ and $w^* = (u^*,v^*)^\top \in C^2_{\mathrm{ub}}(\bbR;\bbR^2)$ such that $\hat{w}(t,x) = (u^*(x - ct),v^*(x-ct))^\top$ is a solution to \cref{eq:fhn}.
Moreover, $u^*$ and $v^*$ converge to $0$ at an exponential rate as $\abs{x} \to \infty$, and the pulse $\hat{w}$ is stable.
We will show how \cref{eq:fhn} fits into our framework, and establish stability of stochastic perturbations of the travelling pulse in a flexible range of spaces at a low regularity level.

There are already multiple frameworks which have treated stochastic perturbations of \cref{eq:fhn} \cite{maclaurin_phase_2023,hamster_stability_2020}.
However, as previously mentioned, these works crucially rely on the Hilbert space structure of $L^2$ to be able to track the pulse and show stability.
This poses significant restrictions on the noise, and does not allow for treatment of equations which are well-posed in $L^p$-spaces for $p \neq 2$.
We will demonstrate that the results from \cref{sec:linsym,sec:nonlinstab} are strong enough to directly show stability in the Bessel space $\cX = H^{s,p}(\bbR;\bbR^2)$ for $p \in [2,\infty)$, $s \in (\tfrac{1}{p},\infty)$, which allows us to treat noise which is much rougher than in previous results.
In the rest of the section we will always assume $p$ and $s$ are in the previously mentioned range.

The condition $p \in [2,\infty)$ ensures that $\cX$ is $2$-smooth (when using an appropriate norm, which we assume from now on).
By Sobolev embedding, the condition $s > \tfrac{1}{p}$ implies that $\cX \hookrightarrow C_{\mathrm{ub}}(\bbR;\bbR^2)$, and it also ensures that $\cX$ is a Banach algebra, i.e.,
\begin{equation}
    \label{eq:algebra}
    \norm{f g}_\cX \leq C \norm{f}_\cX \norm{g}_\cX, \quad f,g \in \cX
\end{equation}
(see \cite[Chapter 4.6]{runst_sobolev_1996}).
We write \cref{eq:fhn} in the form \cref{eq:pde} by taking
\begin{equation*}
    A = 
    \begin{pmatrix}
        \partial_x^2 & 0 \\ 0 & \delta\partial_x^2
    \end{pmatrix}
    , \qquad  F \begin{pmatrix} u \\ v \end{pmatrix}
    =
    \begin{pmatrix}f(u) - v \\ \eps (u - \gamma v)\end{pmatrix}.
\end{equation*}
We begin by establishing higher order regularity of the pulse profile.
\begin{proposition}
    \label{prop:fhnpulsereg}
    We have $u^*,v^* \in H^k(\bbR;\bbR)$ for any $k \in \bbN$.
\end{proposition}
\begin{proof}
    Let $\tilde{S}(t)$ be the semigroup on $L^2(\bbR;\bbR^2)$ generated by $A$ (note that $\tilde{S}(t)$ can be explicitly realized as a convolution with the heat kernel).
    By Duhamel's formula and the fact that $\hat{w}$ solves \cref{eq:fhn}, we have
    \begin{equation}
        \label{eq:fhnduhamelstar}
        \begin{pmatrix}
            u^*(\cdot - c) \\
            v^*(\cdot - c)
        \end{pmatrix}
         = 
        \tilde{S}(1)\begin{pmatrix}
            u^* \\ v^*
        \end{pmatrix}
        + \int_0^1 \tilde{S}(1-t')
        \begin{pmatrix}
            f(u^*(\cdot - ct')) - v(\cdot - ct') \\
            \eps(u^*(\cdot - ct') - \gamma v(\cdot - ct'))
        \end{pmatrix}
        \rd t'.
    \end{equation}
    By continuity of $u^*,v^*,f$ and exponential decay of $u^*,v^*$ we have $u^*,v^*,f(u^*) \in L^2$.
    Since $\norm{\tilde{S}(t)}_{\cL(L^2(\bbR;\bbR^2);H^1(\bbR;\bbR^2))} \leq C t^{-\frac{1}{2}}$, we find from \cref{eq:fhnduhamelstar} that $u^*,v^* \in H^1(\bbR;\bbR)$.
    Furthermore, if $u^*,v^* \in H^k(\bbR;\bbR)$ for some $k \geq 1$, then also $f(u^*) \in H^k(\bbR;\bbR)$ by \cref{eq:algebra}, since $f$ is a polynomial which vanishes at zero.
    Thus, using \cref{eq:fhnduhamelstar} and the estimate $\norm{\tilde{S}(t)}_{\cL(H^k(\bbR;\bbR^2);H^{k+1}(\bbR;\bbR^2))} \leq C_k t^{-\frac{1}{2}}$, we get $u^*,v^* \in H^{k+1}(\bbR;\bbR)$.
    The claim follows by induction on $k$.
\end{proof}
We now introduce the linear operators
\begin{equation*}
    A' = 
    \begin{pmatrix}
        \partial_{xx} + c\partial_x & 0 \\ 0 & \delta\partial_{xx} + c \partial_x
    \end{pmatrix}
    , \qquad  F' = \begin{pmatrix} f'(u^*) & -1 \\ \eps & - \eps \gamma \end{pmatrix}.
\end{equation*}
Combining \cref{prop:fhnpulsereg} with \cref{eq:algebra} and the fact that $f$ is a (cubic) polynomial, we obtain the following regularity properties of $A'$, $F$, and $F'$.
\begin{proposition}
    \label{prop:AFprimereg}
    Let $q \in [2,\infty)$ and $k \in \bbN$.
    The operator $A'$ genereates a semigroup on $H^k(\bbR;\bbR^2)$ and on $L^q(\bbR;\bbR^2)$.
    The nonlinearity $F$ is infinitely Fr\'echet differentiable from $\cX$ to $\cX$.
    The first derivative of $F$ at $w^*$ is given by $F'$, which is also bounded on $H^k(\bbR;\bbR^2)$ and on $L^q(\bbR;\bbR^2)$.
\end{proposition}
We now verify the assumptions formulated in \cref{sec:linsym,sec:nonlinstab}.
We take $\grp = (\bbR,+)$, and define $\Pi$ according to $\Pi(a)w(x) = w(x-a)$.
The corresponding Lie algebra $\alg$ can be identified with $(\bbR,+)$, in which case the exponential map from $\alg$ to $\grp$ acts as the identity.
The action of $\pi$ is given by $\pi(a) = -a\partial_x$.
Since $A$ and $F$ clearly commute with $\Pi(a)$ and translation is strongly continuous in $H^{s,p}(\bbR;\bbR^2)$, we see that \cref{ass:sympde} holds true.
As we have $\hat{w}(t,x) = w^*(x - ct)$, we should take $X = c$ in \cref{eq:uhat}.
For \cref{eq:taylorestimate}, we need to show
\begin{equation*}
    \norm{w^*(\cdot - a) - w^* + a \partial_x w^*}_\cX \leq C \abs{a}^2, \quad a \in \bbR,
\end{equation*}
which follows from a Taylor expansion and the smoothness of $w^*$ established in \cref{prop:fhnpulsereg}.
The remaining statements of \cref{ass:selfsimilarsolution} are easily verified using \cref{prop:fhnpulsereg}, and
\cref{ass:Fregular} follows directly from \cref{prop:AFprimereg}.

The linearization of \cref{eq:fhn} around $w^*$ (in the comoving frame) is formally given by
\begin{equation}
    \label{eq:fhnLstar}
    \cL^* = A' + F' = 
    \begin{pmatrix}
        \partial_x^2 + c \partial_x + f'(u^*) & - 1 \\
        \eps & \delta\partial_x^2 + c \partial_x - \eps \gamma
    \end{pmatrix}.
\end{equation}
From \cref{prop:AFprimereg} and a perturbation argument (see \cite[Chapter 9]{kato_perturbation_1995}), it follows that $\cL^*$ generates a semigroup $S^*(t)$ on $H^k(\bbR;\bbR^2)$ and on $L^q(\bbR;\bbR^2)$ for any $k \in \bbN$ and any $q \in [1,\infty)$. 
By complex interpolation, we find that $S^*(t)$ is also a semigroup on $\cX$.
Hence, \cref{ass:semigroup} is shown except for the boundedness of $S^*(t)$, which will follow from \cref{eq:linstab,eq:Sstarcenter} once \cref{ass:decomposition} has been established.

In \cite{alexander_topological_1990}, \cref{ass:decomposition} has been shown to hold in the space $\cX = C_{\mathrm{ub}}(\bbR;\bbR^2)$.
A slight modification of the argument in \cite[Appendix A.2]{eichinger_multiscale_2022} shows that \cref{ass:decomposition} also holds with $\cX = L^2(\bbR;\bbR^2)$.
In both cases, $P^*_c$ is given by the same spectral projection onto the zero eigenvalue of $\cL^*$.
By a minor modification of the proof of Marcinkiewicz interpolation theorem, (using a smooth cutoff instead of the usual partition $f = \ind_{\abs{f} > a}f + \ind_{\abs{f} \leq a}f$, see \cite{stein_singular_1970}), we see that operators which are bounded on $L^2(\bbR;\bbR^2)$ and $C_{\mathrm{ub}}(\bbR;\bbR^2)$ are also bounded on $L^q(\bbR;\bbR^2)$.
It follows that \cref{ass:decomposition} holds with $\cX = L^q(\bbR;\bbR^2)$ for any $q \in [2,\infty)$.

Now fix any $k \in \bbN$.
Since $P^*_c$ projects onto the space spanned by $\partial_x w^*$, it follows from \cref{prop:fhnpulsereg} that $\norm{P^*_c}_{\cL(L^2(\bbR;\bbR^2);H^k(\bbR;\bbR^2))} < \infty$, which implies that $P^*_c$ (hence, also $P^*_s$) is bounded on $H^k(\bbR;\bbR^2)$.
Furthermore, by the smoothing effect of the Laplacian, we have $\norm{S^*(1)}_{\cL(L^2(\bbR;\bbR^2);H^k(\bbR;\bbR^2))} < \infty$.
Thus, \cref{eq:linstab} self-improves from $\cX = L^2(\bbR;\bbR^2)$ to $\cX = H^k(\bbR;\bbR^2)$, so that \cref{ass:decomposition} holds with $\cX = H^k(\bbR;\bbR^2)$ for any $k \in \bbN$.
By complex interpolation between $L^q(\bbR;\bbR^2)$ and $H^k(\bbR;\bbR^2)$ for appropriate values of $q$ and $k$, we finally find that \cref{ass:decomposition} is satisfied with $\cX = H^{s,p}(\bbR;\bbR^2)$.

With \cref{ass:sympde,ass:selfsimilarsolution,ass:semigroup,ass:decomposition,ass:Fregular} verified, we study stochastic perturbations of \cref{eq:fhn}.
We consider two cases.

\subsubsection{Additive noise}
We consider the following equation:
\begin{equation}
\label{eq:fhnstoch1}
\begin{aligned}
    \rd u(t,x) &= [\partial_x^2 u(t,x)  + f(u(t,x)) - v(t,x)] \rd t + \sigma\sum_{i=0}^{\infty} g_i(t,x)\rd \beta_i(t), \\
    \rd v(t,x) &= [\delta \partial_x^2 v(t,x) + \eps(u(t,x) - \gamma v(t,x))] \rd t,
\end{aligned}
\end{equation}
where $\sigma > 0$, $\beta_i$ is a sequence of independent Brownian motions, and $g_i(t,\cdot)$ is a (deterministic) sequence of functions in $\cX$.
To translate this to the language of \cref{subsec:stochpert}, we take $\cH = \ell^2(\bbN)$ with the usual orthonormal basis $(e_i)_{i \in \bbN}$, and let $G$ be given by $(G(t) e_i)(x) = g_i(t,x)$.
To verify \cref{ass:GHregular}, we only need to make sure that \cref{eq:Gbounded} holds (as $H \equiv 0$ and $\cref{eq:Glipschitz}$ is trivial).
Since $H^{s,p}(\bbR;\bbR)$ has type $2$, we have the embedding $\ell^2(\bbN;H^{s,p}(\bbR;\bbR)) \hookrightarrow \gamma(\ell^2(\bbN);H^{s,p}(\bbR;\bbR))$ \cite[Theorem 9.2.10]{hytonen_analysis_2016}.
Hence, aside from measurability, it suffices to require
\begin{equation}
    \label{eq:fhnnoise1}
    \sup_{t \in [0,\infty]} \sum_{i=0}^{\infty}\norm{g_i(t,\cdot)}_{H^{s,p}(\bbR;\bbR)}^2 < \infty.
\end{equation}
With all the assumptions verified, we obtain from \cref{thm:stochstablong} the following concrete result.
\begin{theorem}
    Let $p \in [2,\infty)$, $s \in (\tfrac{1}{p},\infty)$, let $g_i$ be such that \cref{eq:fhnnoise1} is satisfied, and let $u(t)$ be the solution to \cref{eq:fhnstoch1} with initial condition $u(0,\cdot) = w^*$.
    There exist constants $c,\eps'$ such that we have the estimate
    \begin{equation*}
        \PP[\Big]{ \sup_{t \in [0,T]} \inf_{a \in \bbR} \norm{u(t,\cdot) - w^*(\cdot - a)}_{H^{s,p}(\bbR;\bbR^2)} \geq \eps}
        \leq 12\, T \exp(-c \eps^2 \sigma^{-2}),
    \end{equation*}
    for all $T > 0$ and all $\eps,\sigma$ satisfying $0 < \sigma \leq \eps \leq \eps'$.
\end{theorem}
As remarked before, we can allow the regularity parameter $s$ to be arbitrarily small (as long as $p$ is sufficiently large).
\subsubsection{Multiplicative noise}
We now consider the equation
\begin{equation}
\label{eq:fhnstoch2}
\begin{aligned}
    \rd u(t,x) &= [\partial_x^2 u(t,x)  + f(u(t,x)) - v(t,x)] \rd t + \sigma g(u)(\phi * \rd W(t)), \\
    \rd v(t,x) &= [\delta \partial_x^2 v(t,x) + \eps(u(t,x) - \gamma v(t,x))] \rd t,
\end{aligned}
\end{equation}
where $\sigma > 0$, $\rd W(t)$ is space-time white noise, and $g,\phi \colon \bbR \to \bbR$ are suitable functions.
The symbol $*$ denotes convolution.
Notice that the noise term in \cref{eq:fhnstoch2} is translationally invariant, which is motivated by symmetry and has relevance in many physical applications.
It is well-known that an $L^2(\bbR;\bbR)$-cylindrical Wiener process formally corresponds to space-time white noise.
Therefore, in the language of \cref{subsec:stochpert}, we should take $\cH = L^2(\bbR;\bbR)$,
 and define $G(u)$ via
\begin{equation*}
    (G(t,u)f)(x) = g(u(x))\int_\bbR \phi(y-x)f(y) \rd y.
\end{equation*}
From \cref{prop:Ggammabessel}, it follows that \cref{ass:GHregular} is satisfied if both of the following conditions are met:
\begin{itemize}
    \item $g$ maps $H^{s,p}(\bbR;\bbR)$ into $H^{s,p}(\bbR;\bbR)$, and is locally Lipschitz in $H^{s,p}(\bbR;\bbR)$.
    \item $\phi \in H^{s}(\bbR;\bbR)$.
\end{itemize}
Since $H^{s,p}(\bbR;\bbR)$ is a Banach algebra by our choice of $s,p$, the first condition is satisfied by any smooth function $g \colon \bbR \to \bbR$ which satisfies $g(0) = 0$.
Applying \cref{thm:stochstablong} thus results in the following.
\begin{theorem}
    \label{thm:fhnstoch2}
    Let $p \in [2,\infty)$, $s \in (\tfrac{1}{p},\infty)$, $\phi \in H^s(\bbR;\bbR)$, let $g$ be a polynomial which satisfies $g(0) = 0$, and let $u(t)$ be the solution to \cref{eq:fhnstoch2} with initial condition $u(0,\cdot) = w^*$.
    There exist constants $c,\eps'$ such that we have the estimate
    \begin{equation*}
        \PP[\Big]{ \sup_{t \in [0,T]} \inf_{a \in \bbR} \norm{u(t,\cdot) - w^*(\cdot - a)}_{H^{s,p}(\bbR;\bbR^2)} \geq \eps}
        \leq 12\, T \exp(-c \eps^2 \sigma^{-2}),
    \end{equation*}
    for all $T > 0$ and all $\eps,\sigma$ satisfying $0 < \sigma \leq \eps \leq \eps'$.
\end{theorem}

By taking $p$ large and $s$ small, we see that $\phi$ only needs to be slightly more regular than an $L^2$-function to have orbital stability.

\begin{remark}
    \label{rem:fhnphase}
    Since the symmetry group is commutative, in this case \cref{eq:predictedphase} simplifies to
    \begin{equation*}
        a_t = c t + \cP v_0.
    \end{equation*}
    By the Riesz representation theorem, it also follows that $\cP v_0 = \langle \phi^0,v_0 \rangle_{L^2(\bbR;\bbR^2)}$ for some $\phi^0 \in L^2(\bbR;\bbR^2)$.
    In fact, it can be shown that $\phi^0$ is an eigenfuction of the adjoint of $\cL^*$ with eigenvalue $0$ \cite{hamster_stability_2020}.
\end{remark}

\begin{remark}
    \label{rem:fhnparabolic}
    The proof of \cref{thm:stochstabshort} does not make use of the smoothing effect of the Laplacian in \cref{eq:fhn}.
    Thus, we expect that with a minor modification, we can consider noise which is even more singular in both the additive and multiplicative case.
    We leave this for future research.
\end{remark}

\subsection{Rotating waves in two dimensions}
\label{subsec:rotwave}
Our second example is a localized rotating wave in two dimensions.
Existence and (non)linear stability of such waves has been studied in \cite{beyn_nonlinear_2008,kuehn_stochastic_2022,cohen_rotating_1978,scheel_bifurcation_1998,hagan_spiral_1982,dai_ginzburg_2021}.
Instead of verifying \cref{ass:sympde,ass:selfsimilarsolution,ass:semigroup,ass:decomposition,ass:Fregular} like in the previous example, (which can be done using the results from \cite{beyn_nonlinear_2008}), we focus on computing the symmetry group and phase prediction function, which show how the noncommutativity enters into the phase.

We consider the reaction-diffusion equation
\begin{equation}
    \label{eq:pderotating}
    \rd u(t,x) = [\Delta u(t,x) + f(u(t,x))] \rd t, \qquad x = (x_1,x_2)^\top \in \bbR^2,
\end{equation}
where $u$ takes values in $\bbR^n$, $\Delta$ acts as a Laplacian in all components, and $f \colon \bbR^n \to \bbR^n$ is sufficiently smooth and satisfies $f(0) = 0$.
We assume existence of a smooth and localized \emph{rotating wave} solution of the form $\hat{u}(t,x) = u^*(R_{-\omega t}x)$, where $u^*$ is the wave profile, $\omega > 0$ is the rotation speed, and $R_{\theta}$ is a rotation matrix, i.e.,
\begin{equation*}
    R_{\theta} = \begin{pmatrix}\cos(\theta) & - \sin(\theta) \\ \sin(\theta) & \cos(\theta) \end{pmatrix}, \quad \theta \in \bbR.
\end{equation*}
The natural symmetry group of \cref{eq:pderotating} is the special Euclidean group $\mathrm{SE}(2)$, which is the semi-direct product of the group of translations and the group of rotations around the origin.
$\mathrm{SE}(2)$ can be represented by matrices of the form
\begin{equation*}
    \cT_{x_1,x_2,\theta} = \begin{pmatrix}
        \cos(\theta) & - \sin(\theta) & x_1 \\ 
        \sin(\theta) & \cos(\theta) & x_2 \\
        0 & 0 & 1
    \end{pmatrix},
\end{equation*}
with $x_1,x_2,\theta \in \bbR$ (notice that $(x_1,x_2,\theta) \mapsto \cT_{x_1,x_2,\theta}$ is not injective).
In this case, the group multiplication is just matrix multiplication.
After identifying the vector $(a,b)^\top$ with $(a,b,1)^\top$, we obtain an action of $\cT_{x_1,x_2,\theta}$ on $\bbR^2$ via matrix-vector multiplication.
This induces an action on the space of functions on $\bbR^2$, given by $(\Pi(\cT_{x_1,x_2,\theta}))f(x) = f(\cT_{x_1,x_2,\theta}^{-1}x)$.
Hence, the element $\cT_{x_1,x_2,0}$ acts as a translation by $(x_1,x_2)$, and $\cT_{0,0,\theta}$ acts as counterclockwise rotation around the origin by $\theta$ radians. 
By differentiating with respect to $x_1,x_2,\theta$, we see that the corresponding Lie algebra (which we denote $\mathfrak{se}(2)$) can be represented as the span of the following matrices:
\begin{equation}
    \label{eq:rotX123}
    X_1 = \begin{pmatrix}
        0 & 0 & 1 \\ 0 & 0 & 0 \\ 0 & 0 & 0
    \end{pmatrix},\quad
    X_2 = \begin{pmatrix}
        0 & 0 & 0 \\ 0 & 0 & 1 \\ 0 & 0 & 0
    \end{pmatrix},\quad
    X_3 = \begin{pmatrix}
        0 & -1 & 0 \\ 1 & 0 & 0 \\ 0 & 0 & 0
    \end{pmatrix}.
\end{equation}
Furthermore, we have
\begin{align*}
    e^{tX_1} &= \cT_{x_1,0,0}, & \Pi(e^{tX_1})f(x_1,x_2) &= f(x_1 - t, x_2), & \pi(X_1) &= -\partial_{x_1}, \\
    e^{tX_2} &= \cT_{0,x_2,0}, & \Pi(e^{tX_2})f(x_1,x_2) &= f(x_1, x_2 - t), & \pi(X_2) &= -\partial_{x_2}, \\
    e^{tX_3} &= \cT_{0,0,\theta}, & \Pi(e^{tX_3})f(x) &= f(R_{-t}x), & \pi(X_3) &= -x_1 \partial_{x_2} + x_2 \partial_{x_1},
\end{align*}
where $\partial_{x_1},\partial_{x_2}$ denote partial differentiation with respect to $x_1,x_2$, respectively.
Using these expressions, we find that
\begin{equation}
    \label{eq:rotuhat}
    \hat{u}(t) = \Pi(\cT_{0,0,\omega t})u^* = \Pi(e^{t\omega X_3})u^*.
\end{equation}
Thus, we need to take $X = \omega X_3$ in \cref{ass:selfsimilarsolution}.
From \cref{eq:rotX123}, we also find the following commutation relations:
\begin{equation}
    \label{eq:rotatingcomm}
    [X_1,X_2] = 0,\quad [X_1,X_3] = X_2,\quad [X_2,X_3] = -X_1.
\end{equation}

Abbreviating $\partial_{\psi} = x_1 \partial_{x_2} - x_2 \partial_{x_1}$, we obtain the following expression for the linearization around $u^*$ in the comoving frame:
\begin{equation*}
    \cL^* = \Delta + \omega \partial_{\psi} + f'(u^*).
\end{equation*}
We will now explicitly compute the predicted phase function given by \cref{eq:predictedphase}.
From \cref{eq:defL}, we find
\begin{equation*}
    L Y = [Y,\omega X_3], \quad Y \in \mathfrak{se}(2).
\end{equation*}
Using the ordered basis $(X_1,X_2,X_3)$, the matrix representation of $L$ and $e^{tL}$ (obtained from \cref{eq:rotatingcomm} and by exponentiating) are given by
\begin{equation}
    \label{eq:rotLexpL}
    L = \begin{pmatrix}
        0 & -\omega & 0 \\ \omega & 0 & 0 \\ 0 & 0 & 0
    \end{pmatrix}, \qquad
    e^{tL} = \begin{pmatrix}
        \cos(\omega t) & -\sin(\omega t) & 0 \\ \sin(\omega t) & \cos(\omega t) & 0 \\ 0 & 0 & 1
    \end{pmatrix}.
\end{equation}

From \cref{eq:rotLexpL} we see that $\sigma(L_\bbC) = \cur{0, i \omega, -i\omega}$, so $\cL^*$ has spectrum on the imaginary axis.
The form of $e^{tL}$ clearly shows that it is necessary to include both the translational and the rotational symmetries to have any chance of orbital stability.

It only remains to find an explicit expression for $\cP$.
To do this, we use the approach from \cite[Section 2.2]{beyn_nonlinear_2008} (see also \cite[Section 3.2]{kuehn_stochastic_2022}).
There, it is shown that there exist functions $\phi^1,\phi^2,\phi^3 \in L^2(\bbR^2;\bbR^n)$ such that we have
\begin{equation*}
    P^*_c v_0 = -\langle \phi^1, v_0 \rangle \partial_{x_1} u^*
    - \langle \phi^2, v_0 \rangle \partial_{x_2}u^*
    - \langle \phi^3,v_0 \rangle \partial_{\psi}u^*,
    \quad v_0 \in L^2(\bbR;\bbR^n),
\end{equation*}
where the inner products are taken in $L^2(\bbR^2;\bbR^d)$.
We note that such a representation of $P^*_c$ can also be derived from the Riesz representation theorem.
The functions $\phi^1,\phi^2,\phi^3$ are suitable linear combinations of eigenfunctions of the adjoint of $\cL^*$.
Thus, using for $\mathfrak{se}(2)$ the basis $(X_1,X_2,X_3)$, the map $\cP$ can be written as
\begin{equation}
    \label{eq:rotP}
    \cP v_0 = \begin{pmatrix}
        \langle \phi^1, v_0 \rangle \\ \langle \phi^2, v_0 \rangle \\ \langle \phi^3, v_0 \rangle
    \end{pmatrix}, \quad v_0 \in L^2(\bbR^2;\bbR^n).
\end{equation}
Combining \cref{eq:rotuhat,eq:predictedphase,eq:rotLexpL,eq:rotP}, we arrive at the following expression for the predicted phase:
\begin{equation*}
    \gamma_t = 
    \begin{pmatrix}
        \cos(\omega t) & -\sin(\omega t) & 0 \\ 
        \sin(\omega t) & \cos(\omega t) & 0 \\
        0 & 0 & 1
    \end{pmatrix}
    \exp\bra[\Bigg]{
    B\begin{pmatrix}
        \cos(\omega t) & -\sin(\omega t) & 0 \\ 
        \sin(\omega t) & \cos(\omega t) & 0 \\
        0 & 0 & 1
    \end{pmatrix}
    \begin{pmatrix}
        \langle \phi^1, v_0 \rangle \\ \langle \phi^2, v_0 \rangle \\ \langle \phi^3, v_0 \rangle
    \end{pmatrix}
    },
\end{equation*}
where $B$ is the map sending $(a,b,c)^\top$ to $aX_1 + bX_2 + cX_3$, and $\exp$ denotes the matrix exponential.

\section{Proof of nonlinear stability}
\label{sec:proof}
\subsection{Deterministic stability}
\label{subsec:proofdetstab}
\begin{proof}[Proof of \cref{thm:detstab}]
    Fix $T > 0$, $\delta > 0$.
    We define $v(t) \coloneq u(t) - \hat{u}(t)$.
    Considering that $u(t)$ and $\hat{u}(t)$ both solve \cref{eq:milddet} (with initial values $u^* + v_0$ and $u^*$, respectively),
    we see that
    \begin{equation}
        \label{eq:detv}
        v(t) = S(t,0)v_0 + \int_0^t S(t,t')\bra[\big]{F(\hat{u}(t') + v(t')) - F(\hat{u}(t')) - F'(\hat{u}(t'))v(t')}\rd t'.
    \end{equation}
    We now claim that there exists an $\eps > 0$ such that 
    \begin{equation}
        \label{eq:detstabclaim}
        \norm{v_0}_\cX \leq \eps \implies \norm{v(\cdot)-S(\cdot,0)v_0}_{C([0,T];\cX)} \leq \tfrac{1}{2}\delta\norm{v_0}_\cX.
    \end{equation}
    To see this, we define for $v_0 \in \cX$ the time
    \begin{equation}
        \label{eq:tv0}
        t_{v_0} \coloneq \sup \cur[\big]{ t \in [0,T] : \norm{v}_{C([0,t];\cX)} \leq 3M_2\norm{v_0}_\cX},
    \end{equation}
    where $M_2$ is the constant from \cref{eq:SttestimateStt}.
    Using \cref{ass:Fregular,eq:detv,eq:SttestimateStt,eq:tv0}, we find a constant $C_1$ (independent of $v_0$) such that we have
    \begin{equation*}
        \norm{v(t)} \leq M_2 \norm{v_0}_\cX + M_2 T C_1 (3M_2\norm{v_0}_\cX)^2, \quad t \leq t_{v_0} \wedge T.
    \end{equation*} 
    Choosing $\eps$ small enough based on $C_1,M_2,T$, we get the estimate
    \begin{equation*}
        \norm{v(t)} \leq 2 M_2 \norm{v_0}_\cX, \quad t \leq t_{v_0} \wedge T,
    \end{equation*}
    whenever $\norm{v_0}_\cX \leq \eps$.
    By \cref{eq:tv0} and continuity, this implies $t_{v_0} \geq T$, which leads via \cref{eq:detv} to the further implication that
    \begin{equation*}
        \norm{v(t) - S(t,0)v_0}_\cX \leq M_2 T C_1 (3M_2\norm{v_0}_\cX)^2, \quad t \leq T.
    \end{equation*}
    After choosing $\eps$ even smaller based on $\delta,M_2,T,C_1$, \cref{eq:detstabclaim} follows.
    Since we have the identity
    \begin{equation*}
        u(t) = \hat{u}(t) + S(t,0)v_0 + (v(t) - S(t,0)v_0),
    \end{equation*}
    we can combine \cref{eq:detstabclaim} with \cref{lem:nonlinstep1} to find
    \begin{equation*}
        \norm{u(t) - \Pi(\gamma_t)u^*}_\cX \leq M_3 e^{-at}\norm{v_0}_\cX + C M_1 K_T^2 \norm{v_0}_\cX^2 + \tfrac{1}{2}\delta\norm{v_0}_\cX, \quad t \in [0,T],
    \end{equation*}
    whenever $\norm{v_0}_\cX \leq \eps$.
    Choosing $\eps$ even smaller such that $\eps \leq \delta(2 CM_1K_T^2)^{-1}$, we obtain \cref{eq:detorbstab}.
\end{proof}

\subsection{Stochastic stability, short times}
\label{subsec:proofstochstabshort}
This section contains the proof of \cref{thm:stochstabshort}.
Fix $T > 0$ and $\alpha > 1$.
For $\eps,\sigma > 0$, we define the events
\begin{subequations}
\label{eq:defAB}
\begin{align}
    \label{eq:defA}
    A_{\eps} &\coloneq \cur[\big]{ \omega \in \Omega : \norm{v_0}_\cX \leq \alpha\eps}, \\
    \label{eq:defB}
    B_{\eps,\sigma} & \coloneq \cur[\Big]{ \omega \in \Omega : \sup_{t \in [0,T]} \norm[\big]{\sigma \int_0^t S(t,t')G(t',\hat{u}(t'))\rd W(t')}_\cX \leq \tfrac{1}{4}\eps}.
\end{align}
\end{subequations}
We begin by formulating two lemmas relating to these events.
\begin{lemma}
    \label{lem:stab1}
    There exists a constant $\eps_1' > 0$, independent of $v_0$, such that the inequality
    \begin{equation}
        \label{eq:stabABest}
        \norm{u(t) - \Pi(\gamma_t)u^*}_\cX \leq (M_3 e^{-at}\alpha + \tfrac{1}{2})\eps + \norm{z(t)}_\cX
    \end{equation}
    holds for all $t \in [T \wedge \tau^*)$ and  $\omega \in A_{\eps} \cap B_{\eps,\sigma}$, whenever $0 < \sigma \leq \eps \leq \eps_1'$ ($z$ is as in \cref{prop:asym}, and $M_3$ and $a$ are the constants from \cref{eq:SttestimateSttPs}).
\end{lemma}
\begin{proof}
    By \cref{lem:nonlinstep1} and \cref{eq:defv,eq:expansion}, we have for $t \in [0,\tau^*)$:
    \begin{align*}
            \norm{u(t) - \Pi(\gamma_t)u^*}_\cX &\leq M_3e^{-at}\norm{v_0}_\cX + C M_1 K_T^2 \norm{v_0}^2_\cX \\
             &\quad+ \norm[\big]{\sigma \int_0^t S(t,t')G(t',\hat{u}(t'))\rd W(t')}_\cX + \norm{z(t)}_\cX.
    \end{align*}
    Therefore, for $\omega \in A_{\eps} \cap B_{\eps,\sigma}$ we have by \cref{eq:defAB}
    \begin{equation*}
        \norm{u(t) - \Pi(\gamma_t)u^*}_\cX \leq (M_3e^{-at}\alpha + \tfrac{1}{4})\eps + C M_1 K_T^2 \alpha^2 \eps^2 + \norm{z(t)}_\cX, \quad t \in [0,T \wedge \tau^*).
    \end{equation*}
    The claim now follows by choosing $\eps_1' = (4C M_1 K_T^2 \alpha^2)^{-1}$.
\end{proof}

\begin{lemma}
    \label{lem:stab2}
    There exist constants $\eps_2', c > 0$, independent of $v_0$, such that the inequality
    \begin{equation}
        \label{eq:ztailest}
        \PP[\Big]{\sup_{t \in [0,T \wedge \tau^*)} \norm{z(t)}_\cX \geq \tfrac{1}{2}\eps,\, A_{\eps},\, B_{\eps,\sigma}}
        \leq 3\exp(-c \sigma^{-2})
    \end{equation}
    holds for all $0 < \sigma \leq \eps \leq \eps_2'$.
\end{lemma}

\begin{proof}
For $\eps,\sigma > 0$ we define
\begin{subequations}
\label{eq:tauzv}
\begin{align}
    \tau_v &\coloneq \sup \cur[\big]{t \in [0,T] : \norm{v}_{C([0,t];\cX)} \leq (\alpha M_2 + \tfrac{1}{4})\eps}, \\
    \tau_z &\coloneq \sup \cur[\big]{t \in [0,T \wedge \tau^*) : \norm{z}_{C([0,t];\cX)} \leq \tfrac{1}{2} \eps},
\end{align}
\end{subequations}
where $M_2$ is the constant from \cref{eq:SttestimateStt}.
By continuity and \cref{eq:expansion}, we see that $\tau_v \wedge \tau_z < \tau^*$, so both $u(\tau_v \wedge \tau_z)$ and $z(\tau_v \wedge \tau_z)$ are well-defined.
We also observe that the event $A_{\eps} \cap B_{\eps,\sigma}$ implies $\tau_v = T$ by \cref{eq:defv,eq:defAB}.
Therefore, by continuity of $z$ we have
\begin{equation*}
    \PP[\Big]{\sup_{t \in [0,T \wedge \tau^*)} \norm{z(t)}_\cX \geq \tfrac{1}{2}\eps,\, A_{\eps},\, B_{\eps,\sigma}}
    \leq \PP[\big]{\norm{z(\tau_v \wedge \tau_z)}_{\cX} = \tfrac{1}{2} \eps},
\end{equation*}
so it suffices to estimate the latter probability.

By the mild solution formula for \cref{eq:dz}, we have for $t \in [0,\tau_v \wedge \tau_z]$:
\begin{equation}
\label{eq:defzT1T2}
\begin{aligned}
    z(t) &= \int_0^t S(t,t')\bra[\big]{F(u(t')) - F(\hat{u}(t')) - F'(\hat{u}(t'))[u(t') - \hat{u}(t')]} \rd t' \\
    &\quad+ \int_0^t S(t,t')\sigma \bra[\big]{G(t',u(t')) - G(t',\hat{u}(t'))} \rd W(t') \\
    &\quad+ \int_0^t S(t,t')\sigma^2 H(t',u(t')) \rd t' \\
    &\eqcolon T_1(t) + T_2(t) + T_3(t).
\end{aligned}
\end{equation}
We now set $C_2 \coloneq \alpha M_2 + \tfrac{3}{4}$.
From \cref{eq:uhat,eq:SttestimatePi,eq:expansion,eq:tauzv}, we see that for $t \in [0,\tau_v \wedge \tau_z]$ we have
\begin{align*}
    \norm{u(t)}_\cX &\leq M_1 \norm{u^*}_\cX + C_2 \eps, \\
    \norm{u(t) - \hat{u}(t)}_\cX &\leq C_2\eps.
\end{align*}
By \cref{ass:Fregular,ass:GHregular}, we see that there exists a constant $C_3$, independent of $\eps,\sigma,v_0$, such that the inequalities
\begin{subequations}
\begin{align}
    \label{eq:zFest}
    \norm{F(u(t)) - F(\hat{u}(t)) - F'(\hat{u}(t))[u(t) - \hat{u}(t)]}_\cX &\leq  C_3 C_2^2\eps^2, \\
    \label{eq:zGest}
    \norm{G(t,u(t)) - G(t,\hat{u}(t))}_{\gamma(\cH;\cX)} &\leq C_3 C_2 \eps, \\
    \label{eq:zHest}
    \norm{H(t,u(t))}_\cX &\leq C_3,
\end{align}
\end{subequations}
hold for all $t \in [0,\tau_v \wedge \tau_z]$.
Therefore, by \cref{eq:SttestimateStt} we have the estimates
\begin{align*}
    \norm{T_1(\tau_v \wedge \tau_z)}_\cX &\leq M_2 T C_3 C_2^2 \eps^2, \\
    \norm{T_3(\tau_v \wedge \tau_z)}_\cX &\leq M_2 T C_3 \sigma^2,
\end{align*}
Choosing $\eps_2' = (8M_2 T C_3 C_2^2)^{-1}$ and using \cref{eq:defzT1T2}, this gives (note that $C_2 \geq 1$)
\begin{equation*}
    \norm{z(\tau_v \wedge \tau_z)}_{\cX} \leq \tfrac{1}{4} \eps + \norm{T_2(\tau_v \wedge \tau_z)}_\cX,
\end{equation*}
whenever $0 < \sigma \leq \eps \leq \eps_2'$.
Thus, we have
\begin{equation*}
    \PP[\big]{\norm{z(\tau_v \wedge \tau_z)}_{\cX} = \tfrac{1}{2} \eps} \leq \PP[\big]{\norm{T_2(\tau_v \wedge \tau_z)}_\cX \geq \tfrac{1}{4} \eps},
\end{equation*}
so it only remains to estimate this probability.
Now observe that
\begin{align*}
    \norm{T_2(\tau_v \wedge \tau_z)}_\cX 
    &\leq \sup_{t \in [0,\tau_v \wedge \tau_z]} \norm[\Big]{\int_0^t S(t,t')\sigma \bra[\big]{G(t',u(t')) - G(t',\hat{u}(t'))} \rd W(t')}_\cX \\
    &\leq \sup_{t \in [0,T]} \norm[\Big]{\int_0^t S(t,t') \ind_{[0,\tau_v \wedge \tau_z]}(t')\sigma \bra[\big]{G(t',u(t')) - G(t',\hat{u}(t'))} \rd W(t')}_\cX.
\end{align*}
From \cref{eq:zGest}, we also have
\begin{equation*}
    \norm[\big]{ \ind_{[0,\tau_v \wedge \tau_z]}(\cdot)\sigma \bra[\big]{G(\cdot,u(\cdot)) - G(t,\hat{u}(\cdot))}}_{L^{\infty}(0,T;\gamma(\cH;\cX))} \leq C_3 C_2 \sigma\eps.
\end{equation*}
Thus, by \cref{prop:stochconvtail}, there is a constant $c > 0$, depending only on $S$, such that
\begin{equation*}
    \PP[\big]{\norm{T_2(\tau_v \wedge \tau_z)}_\cX \geq \tfrac{1}{4} \eps} \leq 3 \exp\bra[\big]{-c C_3^{-2} C_2^{-2}T^{-1} \sigma^{-2}}. \qedhere
\end{equation*}
\end{proof}
\begin{proof}[Proof of \cref{thm:stochstabshort}]
For $\eps,\sigma > 0$, we introduce the additional event
\begin{equation*}
    E_{\eps,\sigma} \coloneq \cur[\big]{ \omega \in \Omega: \sup_{t \in [0,T \wedge \tau^*)} \norm{u(t) - \Pi(\gamma_t)\hat{u}(t)}_\cX - M_3\alpha e^{-at}\eps \geq \eps }.
\end{equation*}
Let $\eps'_1$ be the constant obtained from \cref{lem:stab1}, and let $c_2$, $\eps_2'$ be the constants obtained from \cref{lem:stab2}.
Set $\eps' = \eps'_1 \wedge \eps'_2$.
By \cref{lem:stab1,lem:stab2}, we have
\begin{subequations}
\label{eq:EABtails}
\begin{equation}
\begin{aligned}
    \PP{E_{\eps,\sigma} \cap A_{\eps} \cap B_{\eps,\sigma}} &\overset{\cref{eq:stabABest}}{\leq} \PP[\Big]{\sup_{t \in [T \wedge \tau^*)} \norm{z(t)}_\cX \geq \tfrac{1}{2} \eps,\, A_\eps,\, B_{\eps,\sigma}} \\
    &\overset{\cref{eq:ztailest}}{\leq} 3 \exp(-c_2 \sigma^{-2})
\end{aligned}
\end{equation}
for $0 < \sigma \leq \eps\leq \eps'$.
Similarly to how we treated $T_2$ in \cref{lem:stab2}, we find using \cref{prop:stochconvtail} a constant $c_1 > 0$ such that
\begin{equation}
    \PP{\Omega \setminus B_{\eps,\sigma}} \leq 3\exp(-c_1 \sigma^{-2}\eps^2), \quad 0 < \sigma \leq \eps \leq \eps'.
\end{equation}
\end{subequations}
Thus, by a union bound and \cref{eq:EABtails}, we have
\begin{equation*}
    \PP{E_{\eps,\sigma} \cap A_{\eps}} 
    \leq \PP{E_{\eps,\sigma} \cap A_{\eps} \cap B_{\eps,\sigma}} + \PP{\Omega \setminus B_{\eps,\sigma}}
    \leq 3 \exp(-c_1 \sigma^{-2}\eps^2) + 3 \exp(-c_2 \sigma^{-2}),
\end{equation*}
for all $0 < \sigma \leq \eps \leq \eps'$.
The desired estimate \cref{eq:stochstabshort} then follows by taking $c = c_1 \wedge c_2$.
\end{proof}

\subsection{Stochastic stability, long times}
\label{subsec:proofstochstablong}
We now prove \cref{thm:stochstablong}.
First, we show a lemma which allows us to `reset' the phase in \cref{thm:stochstabshort} by performing a coordinate transformation.

\begin{lemma}
    \label{lem:transform}
    Set $v_0 \equiv 0$.
    Let $T > 0$, and let $\gamma_T$ be a $\grp$-valued $\cF_T$-measurable random variable.
    There exist constants $c,\eps',\widetilde{T}$, independent of $T,\gamma_T$, and an $\cF_T$-measurable function $\gamma_{\widetilde{T}}$, such we have the estimate
    \begin{equation*}
        \PP[\big]{ \norm{u(T+\widetilde{T}) - \Pi(\gamma_T\gamma_{\widetilde{T}})u^*}_\cX \geq \eps,\,  
        \norm{u(T) - \Pi(\gamma_T)u^*}_\cX \leq \eps} \leq 6\exp(-c \eps^2 \sigma^{-2}),
    \end{equation*}
    for all $\eps,\sigma$ satisfying $0 < \sigma \leq \eps \leq \eps'$.
\end{lemma}
\begin{proof}
Fix $T > 0$, and set $\widetilde{T} =  a^{-1}\log(2M_1^2 M_3)$, where $a$, $M_1$, and $M_3$ are as in \cref{eq:Sttestimates}.
We use the random transformation $\tilde{u}(t) = \Pi(\gamma_T^{-1})u(T+t)$.
From \cref{ass:sympde} we find that $\tilde{u}$ solves
\begin{align*}
    \rd \tilde{u}(t) &= [A \tilde{u}(t) + F(\tilde{u}(t))]\rd t 
    + \sigma^2 \widetilde{H}(t,\tilde{u}(t)) \rd t 
    + \sigma \widetilde{G}(t,\tilde{u}(t)) \rd \tilde{W}(t) \\
    \tilde{u}(0) &= u^* + \tilde{v}_0,
\end{align*}
where we have introduced
\begin{align*}
    \widetilde{H}(t,f) &= \Pi(\gamma_T^{-1})H(T+t,\Pi(\gamma_T)f), \\
    \widetilde{G}(t,f) &= \Pi(\gamma_T^{-1})G(T+t,\Pi(\gamma_T)f), \\
    \widetilde{W}(t) &= W(T + t), \\
    \tilde{v}_0 &= \Pi(\gamma_T^{-1})u(T) - u^*.
\end{align*}
It can now be seen that \cref{ass:GHregular} still holds when $G,H$ are replaced by $\widetilde{G},\widetilde{H}$ (possibly with worse constants).
By choosing $\alpha = 2M_1^2$, and $T = \widetilde{T}$ in \cref{thm:stochstabshort}, we thus find $c,\eps' > 0$, and $\gamma_{\widetilde{T}}$ such that the estimate
\begin{equation}
    \label{eq:resetproofest}
    \PP[\big]{ \norm{\tilde{u}(\widetilde{T}) - \Pi(\gamma_{\widetilde{T}})u^*}_\cX \geq 2\eps,\,  \norm{\tilde{v}_0}_\cX \leq 2M_1^2 \eps} \leq 6\exp(-c \eps^2 \sigma^{-2})
\end{equation}
holds for all $0 < \sigma \leq \eps \leq \eps'$.
Note that in \cref{thm:stochstabshort}, $\gamma_{\widetilde{T}}$ is obtained purely from $\tilde{v}_0$ (via \cref{eq:predictedphase}), so that $\gamma_{\widetilde{T}}$ is $\cF_T$-measurable.
Using \cref{eq:SttestimatePi}, we also find
\begin{align*}
    \norm{\tilde{v}_0}_\cX &= \norm{\Pi(\gamma_T^{-1})\bra{ u(T) - \Pi(\gamma_T)u^* }}_\cX \\
    &\leq M_1 \norm{ u(T) - \Pi(\gamma_T)u^* }_\cX,
\end{align*}
and
\begin{align*}
    \norm{u(T+\widetilde{T}) - \Pi(\gamma_T\gamma_{\widetilde{T}})u^*}_\cX
    &= \norm{\Pi(\gamma_T) (\tilde{u}(\widetilde{T}) - \Pi(\gamma_{\widetilde{T}})u^*)}_\cX \\
    &\leq M_1 \norm{\tilde{u}(\widetilde{T}) - \Pi(\gamma_{\widetilde{T}})u^*}_\cX.
\end{align*}
Combining this gives
\begin{align*}
    \PP[\big]{ \norm{u(T+\widetilde{T}) &- \Pi(\gamma_{T}\gamma_{\widetilde{T}})u^*}_\cX \geq 2 M_1 \eps,\,  
        \norm{u(T) - \Pi(\gamma_T)}_\cX \leq 2 M_1 \eps} \\
    &\leq \PP[\big]{\norm{\tilde{u}(\widetilde{T}) - \Pi(\gamma_{\widetilde{T}}) u^*}_\cX \geq 2\eps,\,  \norm{\tilde{v}_0}_\cX \leq 2M_1^2 \eps} \\
    \overset{\eqref{eq:resetproofest}}&{\leq} 6\exp(-c \eps^2 \sigma^{-2}),
\end{align*}
for every $0 < \sigma \leq \eps \leq \eps'$. The result follows by rescaling $\eps',c$ based on $M_1$.
\end{proof}

\begin{proof}[Proof of \cref{thm:stochstablong}]
Let $c, \eps', \widetilde{T}$ be as in \cref{lem:transform}.
Using \cref{lem:transform}, we inductively find a sequence $\gamma_n$ of $\cF_{n\widetilde{T}}$-measurable $\grp$-valued random variables such that
\begin{equation*}
    \PP[\big]{ \norm{u((n+1)\widetilde{T}) - \Pi(\gamma_{n+1})u^*}_\cX \geq \eps,\,  
        \norm{u(n\widetilde{T}) - \Pi(\gamma_n)u^*}_\cX \leq \eps} \leq 6\exp(-c \eps^2 \sigma^{-2}),
\end{equation*}
for all $0 < \sigma \leq \eps \leq \eps'$ and all $n \in \bbN$.
Iterating this estimate and using a union bound, we find
\begin{equation*}
    \PP[\Big]{ \max_{k \leq n} \norm{u(k\widetilde{T}) - \Pi(\gamma_k)u^*}_\cX \geq \eps}
    \leq 6n \exp(-c \eps^2 \sigma^{-2}).
\end{equation*}
Applying \cref{thm:stochstabshort} on the intermediate intervals $[n\widetilde{T},(n+1)\widetilde{T}]$ then gives the result.
\end{proof}

\appendix
\section{Radonifying operators}
\label{app:radonifying}
Throughout this section, we write $L^p$ as a shorthand for $L^p(\bbR;\bbR)$, and likewise for $H^{s,p}$ and $W^{k,p}$. 
For measurable $u,\phi,f \colon \bbR \to \bbR$, we define the trilinear operator $G$:
\begin{align}
    \label{eq:tripleG}
    G(u,\phi,f)(x) = u(x)\int_\bbR \phi(y-x)f(y) \rd y.
\end{align}
\begin{lemma}
    \label{lem:gammaG}
    Let $p \in [1,\infty)$.
    There exists a constant $C$ such that the estimate
    \begin{equation}
        \label{eq:Ggammalebesgue}
        \norm{G(u,\phi,\cdot)}_{\gamma(L^2;L^p)} \leq C \norm{u}_{L^p}\norm{\phi}_{L^2},
    \end{equation}
    holds for all $u \in L^p$, $\phi \in L^2$.
\end{lemma}
\begin{proof}
    Let $e_k$ be an orthonormal basis of $L^2$.
    By the $\gamma$-Fubini isomorphism (see \cite[Theorem 9.4.8]{hytonen_analysis_2016}) and Parseval's identity, we have
    \begin{align*}
        \norm{G(u,\phi,\cdot)}_{\gamma(L^2;L^p)}^p
        &\leq C \int_\bbR \bra[\Big]{\sum_{k=0}^{\infty} \abs{G(u,\phi,e_k)(x)}^2}^{\frac{p}{2}}\rd x \\
        &= C \int_\bbR \bra[\Big]{\sum_{k=0}^{\infty} \abs{u(x)\int_\bbR \phi(y-x)e_k(y) \rd y}^2}^{\frac{p}{2}}\rd x \\
        &= C \int_\bbR \abs{u(x)}^p \norm{\phi(\cdot - x)}_{L^2}^p \rd x \\
        &= C \norm{u}_{L^p}^p \norm{\phi}_{L^2}^p. \qedhere
    \end{align*}
\end{proof}
\begin{proposition}
    \label{prop:Ggammabessel}
    Let $p \in (1,\infty)$, $s \in [0,\infty)$.
    There exists a constant $C$ such that the estimate
    \begin{equation}
        \label{eq:GgammaBessel}
        \norm{G(u,\phi,\cdot)}_{\gamma(L^2;H^{s,p})} \leq C \norm{u}_{H^{s,p}}\norm{\phi}_{H^{s}},
    \end{equation}
    holds for all $u \in H^{s,p}$, $\phi \in H^{s}$.
\end{proposition}
\begin{proof}
    The case $s = 0$ is just \cref{eq:Ggammalebesgue}.
    By differentiating \cref{eq:tripleG}, we find the identity $\partial_x G(u,\phi,f) = G(\partial_x u,\phi,f) - G(u,\partial_x \phi,f)$.
    It is well-know that for $p \in (1,\infty)$, the Bessel space $H^{k,p}$ coincides with the classical Sobolev space $W^{k,p}$ for $k \in \bbN$, see \cite[Chapter 3]{stein_singular_1970}.
    Therefore, \cref{eq:GgammaBessel} with $s=k+1$ follows from the case $s=k$ for every $k \in \bbN$.
    The case $s \in [0,\infty) \setminus \bbN$ follows by complex bilinear interpolation~\cite[Theorem 4.4.1]{bergh_interpolation_1976}.
    Here, we use \cite[Theorem 9.1.25]{hytonen_analysis_2016} to interpolate between the spaces $\gamma(L^2;H^{k,p})$ for different values of $k$.
\end{proof}

\printbibliography

\end{document}